\documentclass[a4paper,12pt]{article}
\usepackage{graphicx}
\usepackage{amsmath,amssymb,cite,amsthm,caption}
\usepackage{enumitem}
\usepackage{color}

\usepackage{fullpage}

\usepackage[labelformat=simple]{subcaption}

\renewcommand\theparagraph{\alpha{paragraph}}

\newtheorem{theorem}{Theorem}
\newtheorem{lemma}{Lemma}

\theoremstyle{definition}
\newtheorem{remark}{Remark}

\numberwithin{equation}{section}
\numberwithin{figure}{section}

\title{Dynamics of Discrete Time Systems with a Hysteresis Stop Operator}

\author{Maxim Arnold\thanks{University of Texas at Dallas, TX, USA}\and
Nikita
Begun\thanks{Free University of Berlin, Germany and Saint Petersburg State
University, Russia} \and Pavel Gurevich\thanks{Free University of Berlin,
Germany and Peoples' Frienship University of Russia, Russia}\and Eyram
Kwame\footnotemark[1]\and Harbir Lamba\thanks{George Mason
University,
VA, USA} \and Dmitrii~Rachinskii\footnotemark[1]}

\date{}

\begin{document}

\maketitle
\newcommand{\slugmaster}{%
	\slugger{siads}{xxxx}{xx}{x}{x--x}}

\begin{abstract}
We consider a piecewise linear two-dimensional dynamical system that couples a linear equation with the so-called {\em stop} operator.
Global dynamics and bifurcations of this system are studied depending on two parameters.
The system is motivated by modifications to general-equilibrium
macroeconomic models
that attempt to capture the frictions and memory-dependence of
realistic economic agents.
\end{abstract}

\section{Introduction}

The {\em stop} operator was proposed by L.~Prandtl as an elementary model of quasistatic elastoplasticity \cite{prandtl}, see
Fig.~\ref{fig:0a}. It presents a simple example of a {\em rate-independent}
operator
 with local {\em memory} \cite{visintin},
and, as such, is used as an elementary building block for important models of {\em hysteresis} phenomena such as the Prandtl--Ishlinskii operator \cite{K}, the Preisach operator \cite{KP}, and their generalizations \cite{mayergoyz1
}. Applications of these {\em nonsmooth} operators
include modeling friction \cite{
ruderman}, elastoplastic materials \cite{pi0
}, magnetic hysteresis \cite{mayergoyz1},
fatigue and damage counting \cite{Mic,Ryc}, constitutive laws of {\em smart} materials \cite{
model,iyertan,pra
}, sorption hysteresis \cite{d13,d16,soilstab,
sander}, and phase transitions \cite{BS}. More recent applications range
from biology and medicine \cite{grant,med1} to economics and finance
\cite{Lam1,cross1,
nsf1}.
On the other hand, {\em stop} can also be viewed as a solution operator of a simple variational inequality
describing the Moreau sweeping process with rigid characteristic in one dimension \cite{moro}, see Fig.~\ref{fig:0b}.

\begin{figure}
\centering
\begin{subfigure}{.4\textwidth}
	\vspace{0.38 cm}
  \includegraphics[width=\textwidth]{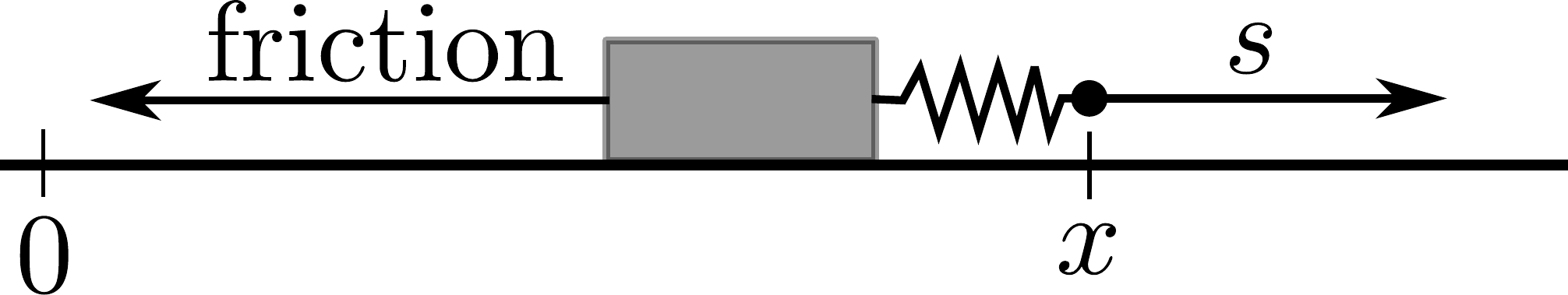}
  \caption{}
  \label{fig:0a}
\end{subfigure}\qquad
\begin{subfigure}{.4\textwidth}
  \includegraphics[width=\textwidth]{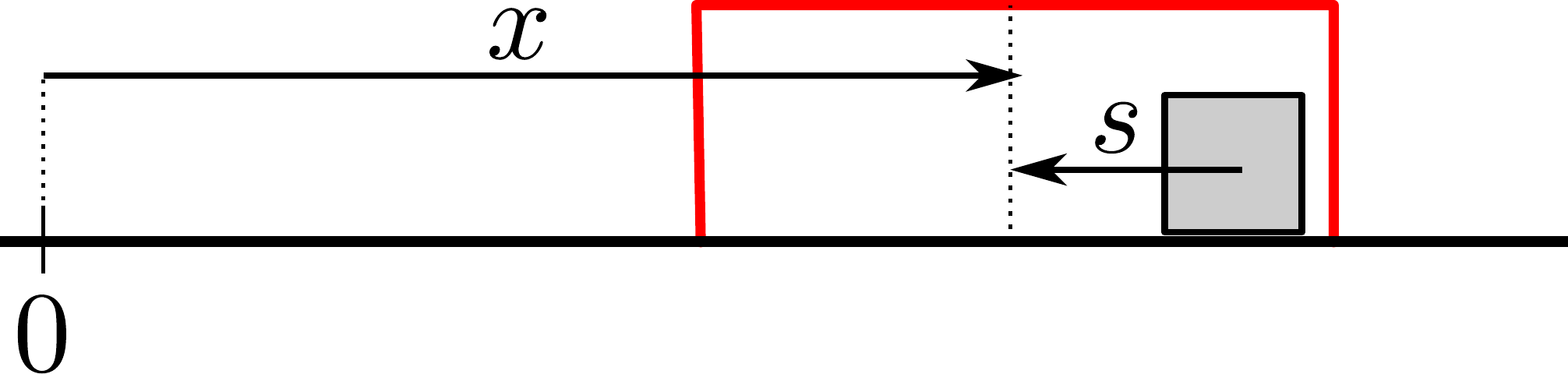}
  \caption{}
  \label{fig:0b}
\end{subfigure}
\caption{Interpretations of the stop operator. (a): Schematic of the
	Prandtl's model of quasi-static elastoplasticity: the box is not moving
	unless
	the absolute value of the force $s$ of the ideal spring reaches the
	maximal value $\rho$ of friction. (b): Schematic of the Moreau
	sweeping process with rigid characteristic in one dimension.
	The position $x$ of the center of the outer frame is the input; the relative
	position $s$ of the center of the frame  with respect to
	the center of the box is the output. The frame moves and drags the
	box.
	}
\label{fig:test}
\end{figure}

Modeling of closed systems that exhibit hysteresis typically leads to differential equations
which include the above nonsmooth operators. Dynamics of these systems
have been analyzed with various techniques including topological degree methods \cite{hopf1, hopf2, chapter,ten}, differential inclusions \cite{kunze},
switched systems \cite{Ast}, and energy considerations using the dissipative property of hysteresis \cite{
dissip2}.
As most of these models are motivated by engineering and physics applications,
they are naturally formulated in continuous time setting.
Discrete time systems with hysteresis operators have received little attention and were studied mostly
in the context of numerical discretizations of continuous systems. However,
the discrete time modeling is typical for certain applications, e.g., in economics, and one can expect that discrete time models motivated by
such applications can exhibit interesting dynamical scenarios
when nonsmooth hysteresis terms are included.

In this paper, we consider an example of a simple discrete time system which consists of a linear
scalar equation coupled with the one-dimensional {\em stop} operator. This system can be equivalently written as a two-dimensional
piecewise linear map. It has multiple equilibrium points which form a segment in the phase space.
We present analysis of global dynamics and bifurcations depending on two parameters.
In particular, the global attractor can consist of two semi-stable equilibrium points,
a segment of stable equilibrium points, a segment of unstable equilibrium points,
a $2$-periodic orbit, a $2$-periodic orbit and two semi-stable equilibrium
points, and, in a critical case,
a two-dimensional set of equilibrium and $2$-periodic points.
For a certain open set of parameter values, the system possesses infinitely many unstable periodic orbits.

The paper is organized as follows. In the next section, we present the main
results.
In Section \ref{sec:discussion}, a motivating economics example is
discussed.
In a standard setting of a general-equilibrium macroeconomic model,
we propose modeling stickiness in agent's expectations by the {\em play}
operator dual to the {\em stop}.
This leads us to a four-dimensional system containing the {\em stop} operator, and we
present a few numerical examples of its dynamics.
The two-dimensional system discussed in Section \ref{sec:main} can be
considered as a simple prototype counterpart of this
higher dimensional economic model.
The last section contains the proofs.

\section{Main results}
\label{sec:main}
Let $s_0\in[-1,1]$ and let $\{x_n\}$, $n\in \mathbb{N}_0$, be a real-valued sequence.
The {\em stop} operator $S$ maps a pair $s_0,\, \{x_n\}$ to a sequence defined by the formula
$$
s_{n+1}=\Phi(s_n+x_{n+1}-x_n),\quad n\in\mathbb{N}_0,
$$
where
\begin{equation}\label{eq1}
\Phi(\tau)=\left\{\begin{array}{rll} -1 & {\rm if} & \tau<-1,\\
\tau & {\rm if} & |\tau|\le 1, \\
1 & {\rm if} & \tau>1,
\end{array}
\right.
\end{equation}
see Fig.~\ref{fig:0c}. Here $s_0$ is called the initial state, $\{x_n\}$ is called the input, and $\{s_n\}$ is called the output (or, the variable state) of the stop operator.
The operator that maps the pair
$s_0$, $\{x_n\}$ to the
sequence $\{x_n-s_n\}$ is called the \emph{play} operator (see
Fig.~\ref{fig:0d}).

\begin{figure}[ht!]
	\centering
	\begin{subfigure}{.48\textwidth}
	\includegraphics[width=\textwidth]{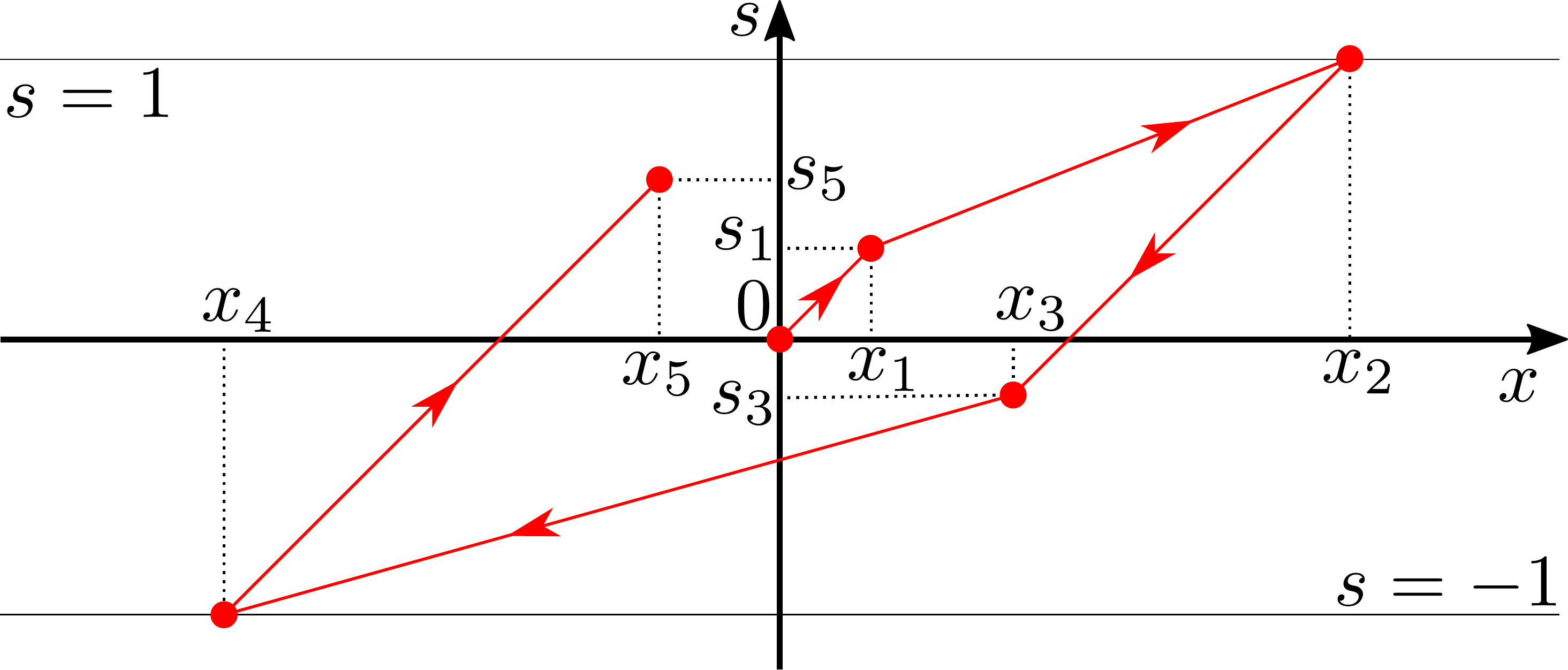}
		\caption{}\label{fig:0c}
	\end{subfigure}\quad
	\begin{subfigure}{.48\textwidth}
		\vspace{0.1 cm}
	\includegraphics[width=\textwidth]{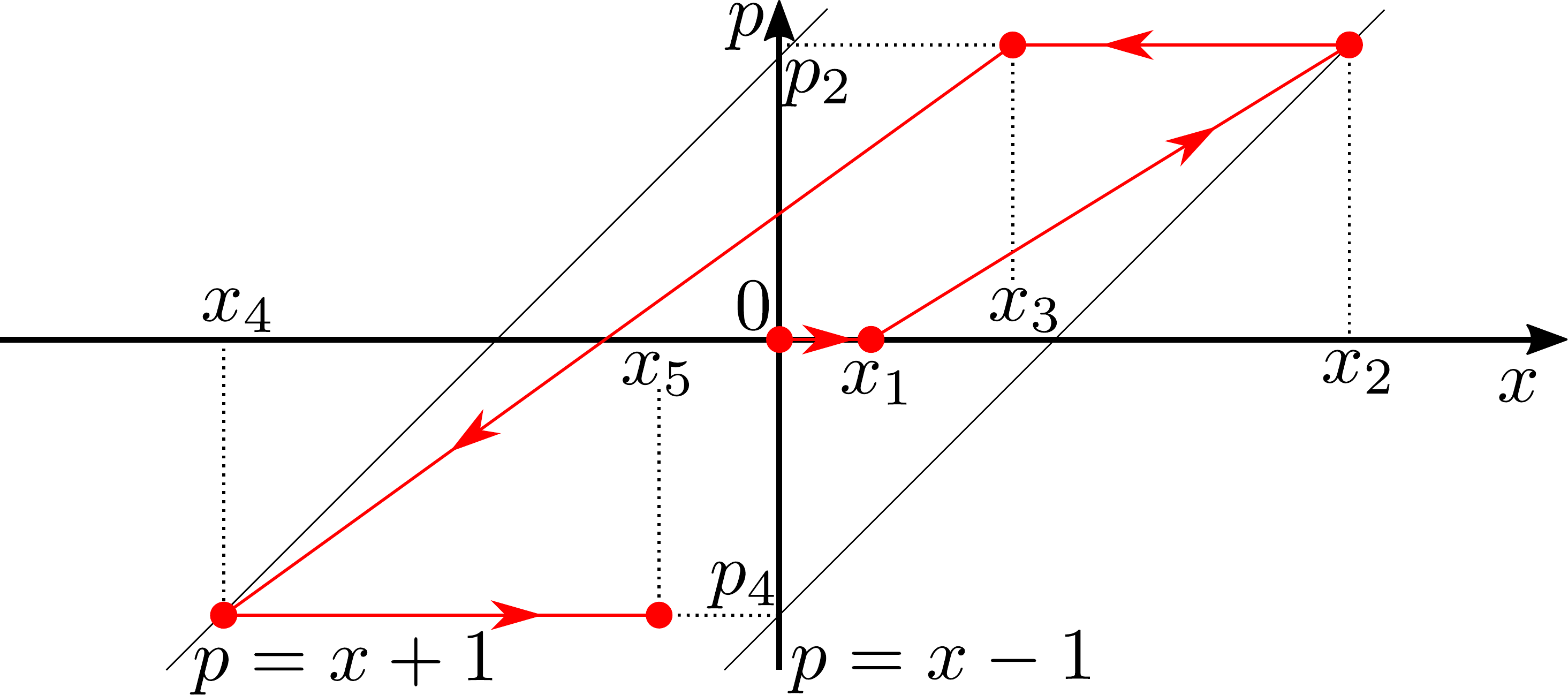}
		\caption{}\label{fig:0d}
	\end{subfigure}
	\caption{Interpretations of the stop and play operators. (a): An
	example of the input-output sequence of the stop
		operator: $(0,0)=(x_0,s_0)$, $ (x_1,s_1)$, $(x_2,1)$, $(x_3,s_3)$,
		$(x_4,-1)$, $(x_5,s_5)$. (b): An
		example of the input-output sequence of the play
		operator: $(0,0)=(x_0,p_0)$, $(x_1,p_0)$,
		$(x_2,p_2)$, $(x_3,p_2)$, $(x_4,p_4)$,
		$(x_5,p_4)$.}
	\label{fig:PlayStop}
\end{figure}

Coupling the output of the stop operator with the input sequence via a linear
transformation with real-valued coefficients  $\lambda$ and $a$, we consider
the dynamical system
\begin{equation}\label{eq2}
\begin{cases}
x_{n+1}=\lambda x_n+ as_n,\\
s_{n+1}=\Phi(s_n+x_{n+1}-x_n)
\end{cases}
\end{equation}
with $n\in \mathbb{N}_0$ on the strip $L=\left\{(x,s): x \in \mathbb{R}, s \in
[-1,1]\right\}$. From hereon we assume
that $|\lambda|<1$. This inequality ensures that all the trajectories of system \eqref{eq2} are
bounded.

It is easy to see that the equilibrium points of system \eqref{eq2}
form the segment
\begin{equation}\label{eq3}
EF=\left\{(x,s): \  x=\frac{a s}{1-\lambda},\ -1\leq s\leq 1\right\}
\end{equation}
with the end points
$$E=(x^*,1)=\left(\frac{a}{1-\lambda},1\right), \quad
F=(-x^*,-1)=\left(-\frac{a}{1-\lambda},-1\right).
$$

 We use the standard notion of stability and
instability (in the Lyapunov
sense) for equilibria and periodic orbits.
We will also say that an equilibrium point $(x_e,s_e)$ of system \eqref{eq2}
is {\em semi-stable} if there are open sets $U_1, U_2\subset \{(x,s):
|s|<1\}$ such that $(x_e,s_e)$ belongs to their boundaries and
simultaneously:
\begin{itemize}
	\item for every $\varepsilon>0$ there is a $\delta>0$ such that any
	trajectory starting from the $\delta$-neighborhood of the equilibrium
	point
	$(x_e,s_e)$ in the set $U_1$ belongs to the $\varepsilon$-neighborhood
	of $(x_e,s_e)$ for all positive $n$;
	
	\item there is an $\varepsilon_0>0$ such that any trajectory starting in
	$U_2$ leaves the $\varepsilon_0$-neighborhood of the equilibrium
	$(x_e,s_e)$ after a finite number of iterations.
\end{itemize}

Our main result consists in the classification of the long time behavior
for  the orbits of system \eqref{eq2}.
Dynamics of system \eqref{eq2} depends on the values of the
parameters $\lambda$ and $\beta=\lambda+a$ as described in the
Theorem \ref{Theorem:1} (see also the
diagram in Fig.~\ref{fig:2b}).

\begin{figure}[htb!]
\begin{center}
\includegraphics[width=0.8\textwidth]{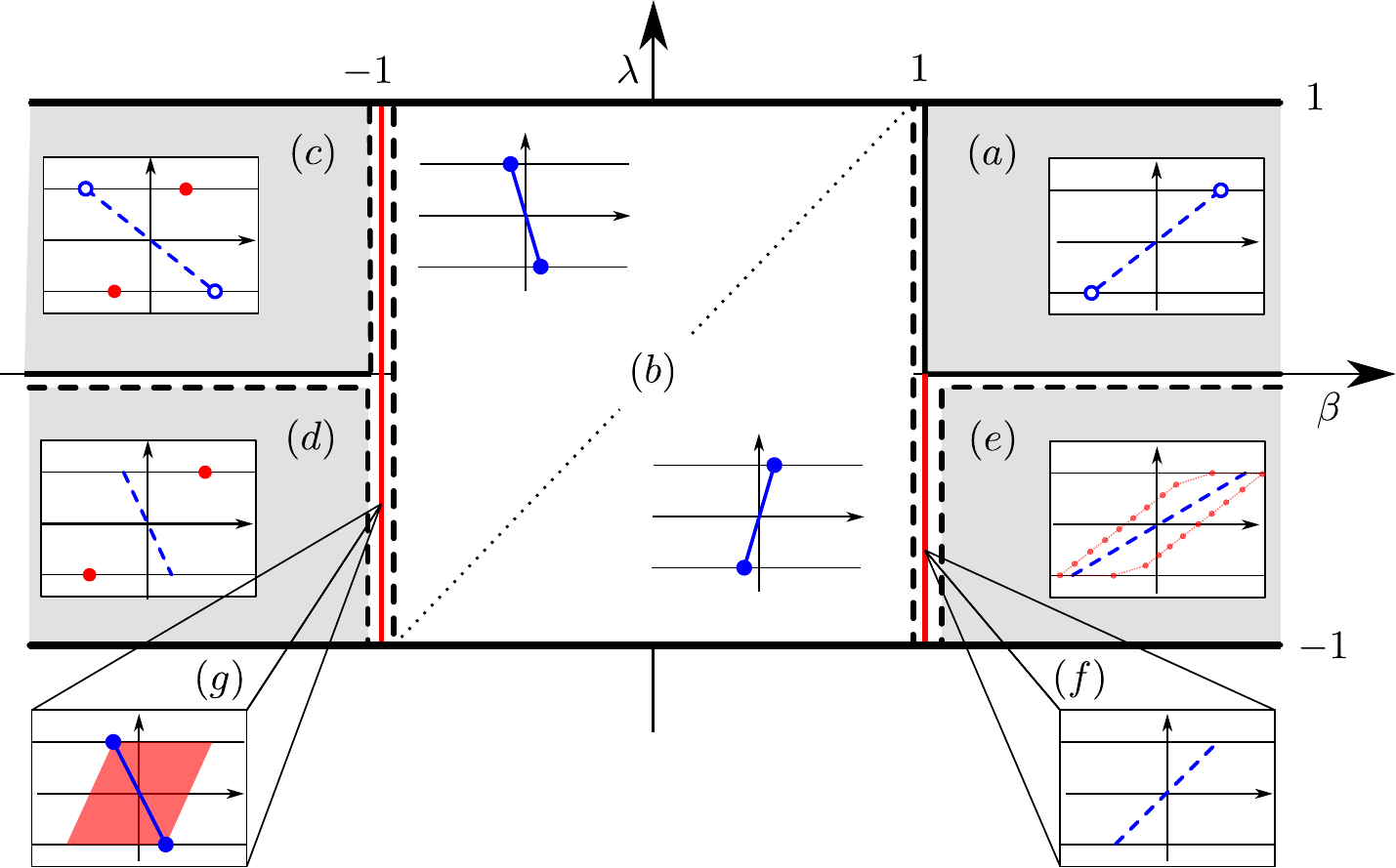}
\vspace{-1mm}
\caption{Bifurcation diagram. The segment $EF$
of the fixed points is shown by the blue line. Stable fixed
points are denoted by the solid line, unstable fixed points are shown by
the dashed line.  Stable end-points of
$EF$ are shown as filled blue discs; semi-stable points are denoted by
empty blue discs; in the unstable case, no special notation is used.  The dotted
line in case \ref{item_a} corresponds to the set
of parameters
leading to the infinite slope of the line $EF$.
Periodic points are shown in red. Filled red discs in cases \ref{item_c}
and \ref{item_d} correspond to the stable $2$-periodic orbit $\pm Q$;
the red parallelogram in case \ref{item_g} consists of stable
$2$-periodic orbits.
Case
\ref{item_e} corresponds to complex dynamics when the system has periodic orbits with arbitrary large periods (see Theorem
\ref{Theorem:2}).  One such orbit is sketched on the
diagram. In the critical case~\ref{item_f}, the segment
$EF$ attracts all the trajectories.}
\label{fig:2b}
\end{center}
\end{figure}

\begin{theorem}
	\label{Theorem:1} Let $\beta=\lambda+a$ and
$|\lambda|<1$.
\begin{enumerate}[label=\rm{(\emph{\alph*})}]

\item\label{item_b} If $\lambda\ge 0$, $\beta\ge 1$, then the equilibrium points
$E$
and $F$
are semi-stable and all the other equilibrium points are unstable. Each non-equilibrium trajectory either converges to $E$ or to
$F$.
\item \label{item_a} If $|\beta|<1$, then all the equilibrium points are
stable and
each trajectory of system \eqref{eq2} converges to an equilibrium point.

\item \label{item_c} If $\lambda\ge 0$, $\beta<-1$, then the points
$E$
and $F$ are semi-stable, all the other equilibrium points are unstable, and
there exists a
stable $2$-periodic orbit
\begin{equation}\label{per2}
\pm Q=\left(\mp\frac{a}{1+\lambda},\pm 1\right).
\end{equation}
Each non-equilibrium trajectory either converges to $E$ or to $F$ or to the
orbit \eqref{per2}.

\item \label{item_d} If $\lambda<0$, $\beta<-1$, then all the
equilibrium
points are unstable. Each non-equilibrium trajectory converges to the stable
$2$-periodic orbit \eqref{per2}.

\item \label{item_e} If $\lambda<0$, $\beta>1$, then all the
equilibrium
points are unstable. System \eqref{eq2} has periodic orbits
of all sufficiently large periods.
At most one periodic orbit is stable.

\item\label{item_f} If $\lambda<0$, $\beta=1$, then all the equilibrium
points are unstable. Each trajectory either ends up at $E$ or at $F$, or
converges to the segment $EF$.

\item \label{item_g} If $\beta=-1$, then all the equilibrium points are
stable. The parallelogram
$$\Sigma=\left\{(x,s):\
2\frac{(1-\lambda)x-a}{1-\lambda+a}+1 \leq s\leq
2\frac{(1-\lambda)(x-1)}{1-\lambda+a}+1,\ |s|\leq 1\right\}$$
with the vertices $E, F$, $Q=(1,1)$ and  $-Q=(-1,-1)$
consists
of stable
$2$-periodic orbits and the diagonal $EF$ of fixed points. Every
non-equilibrium trajectory converges either to one of the equilibrium
points $E$ or $F$, or to a $2$-periodic orbit in the parallelogram
$\Sigma$.
\end{enumerate}
\end{theorem}

The existence of infinitely many periodic orbits in case \ref{item_e}
may indicate the presence of a global strange attractor or a chaotic attractor co-existing
with the stable periodic orbit. More detailed analysis of this case will be a
subject of future work.

Theorem \ref{Theorem:1} describes several bifurcation scenarios. In
particular, the period doubling scenario is interesting
because the stable $2$-periodic orbit, which exists for $\beta<-1$ (see
cases
\ref{item_c} and \ref{item_d}) is not close
to any equilibrium point (as would be typical for smooth systems).
Let us consider $\beta$ as a decreasing bifurcation parameter.
When this parameter crosses the value $-1$, the equilibrium points of the segment $EF$, which are stable for
$\beta\in (-1,1)$ (see case \ref{item_a}), destabilize and the $2$-periodic
orbit \eqref{per2}
appears away from the segment $EF$.
This transition is accompanied by the creation of the parallelogram
$\Sigma$ filled with $2$-periodic orbits
at the critical value $\beta=-1$. This parallelogram is spanned
by the $2$-periodic orbit $\pm Q=(\pm 1,\pm 1)$ and the equilibrium
points
$E$, $F$ (case \ref{item_g}).

Assume that $\lambda<0$. When the parameter $\beta$ increases and
crosses the value $1$, the
equilibrium points destabilize and infinitely many periodic orbits appear
(see case \ref{item_e}).
Dynamics for the critical value $\beta=1$ is described by case
\ref{item_f}. The following theorem complements case \ref{item_e} of
Theorem
\ref{Theorem:1}.

\begin{theorem}\label{Theorem:2}
Assume that the conditions of case \ref{item_e} of Theorem
\ref{Theorem:1}
hold
and hence system \eqref{eq2} has infinitely many periodic orbits, of
which at most one is stable.  Then the relation
$(\lambda,\beta)\in \Omega_k$ with
\begin{equation}\label{eqn:C66'}
\Omega_k=\left\{(\lambda,\beta):\
\dfrac{\beta^k-1}{\beta-1}\le-\frac1{\lambda}<\beta^k, \ \beta>1,
-\frac1{\lambda}>1\right\},
\end{equation} where $k \in \mathbb{N}$,
ensures that system \eqref{eq2}  has a unique stable $(2k+2)$-periodic
orbit. If
$
(\lambda,\beta)\not \in \bigcup\limits_{k=1}^\infty \Omega_k,
$
then all the periodic orbits are unstable.
\end{theorem}

%

\begin{remark}
	{\rm
\ The domains $\Omega_k$ of existence of stable periodic orbits with
different periods do not intersect (see
Fig.~\ref{fig:6}).}
\end{remark}

\begin{remark} {\rm \ It will follow from the proof of
Theorem~\ref{Theorem:2}
that if $(\lambda,\beta)$ belongs to the interior of $\Omega_k$ for some
$k$, then the corresponding stable periodic orbit is asymptotically stable.}
\end{remark}

\begin{figure}[ht]
\begin{center}
\includegraphics[width=0.6\textwidth]{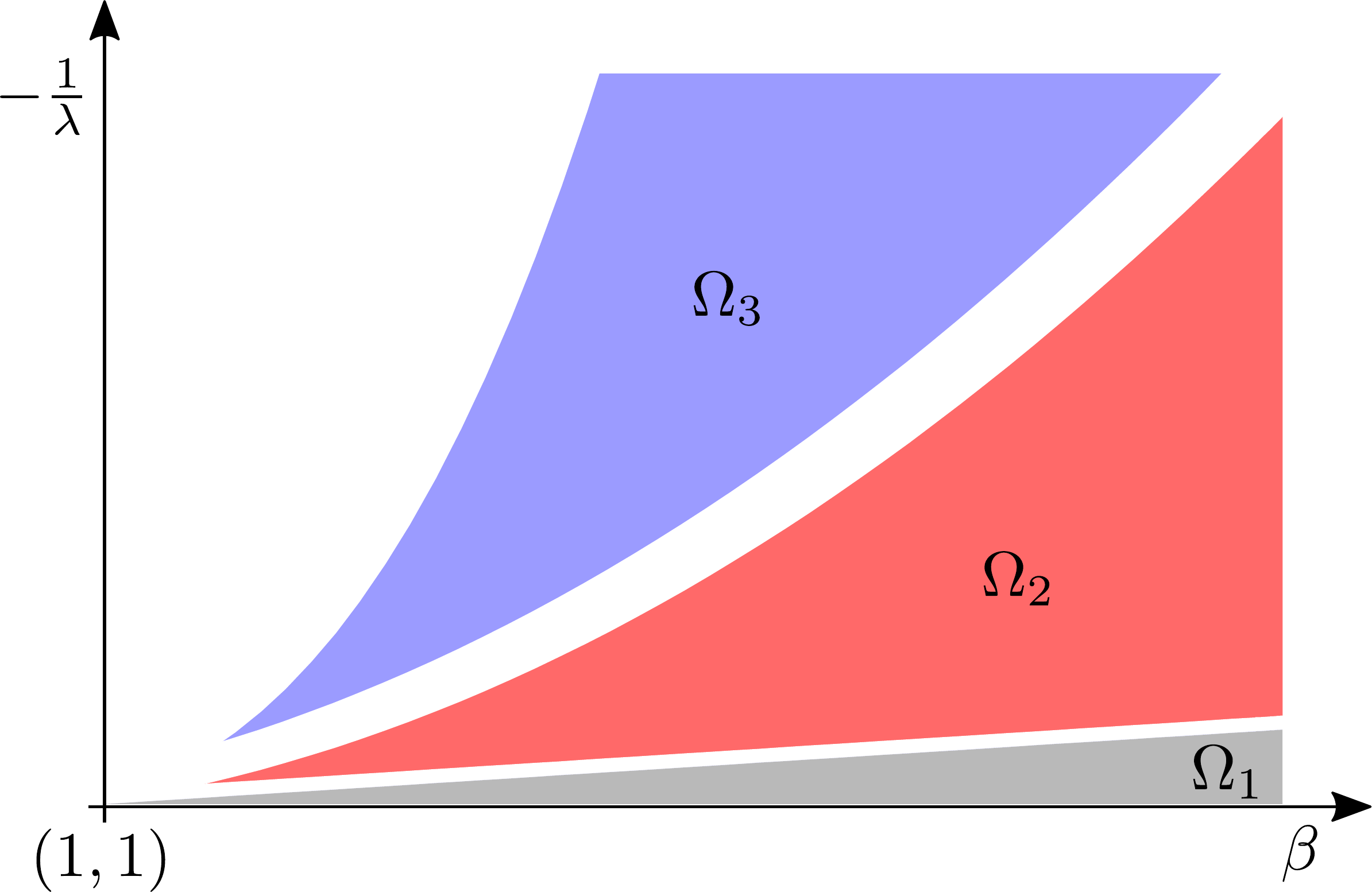}
\caption{Domains $\Omega_k$ of existence of a (unique) stable periodic
orbit of period $2k+2$
in the coordinates $\beta=\lambda+a>1$ and $-1/\lambda>1$ for
case \ref{item_e} of Theorem \ref{Theorem:1}.}\label{fig:6}
\end{center}
\end{figure}

\section{Discussion}\label{sec:discussion}
Hysteresis effects, which are well known in engineering and physics, have become a topic of interest in
other disciplines. In economics, hysteresis has been well documented in the relationship between
the output of the economy and unemployment rate \cite{cross0}. Hysteresis
has been also closely associated with other stylized facts  such as {\em
path dependence} \cite{pathd,g}, {\em stickiness} of prices and information
\cite{sticky3,sticky4,sticky1}
, and {\em heterostasis} (multiplicity of equilibria) \cite{Goecke-play1}
that describe empirical economic data. An attempt to obtain quantitative models of these empirical observations naturally motivated the
use of the play operator and more complex models of hysteresis developed in physics in the economic context.
For example, the play operator was shown to produce a good model of the dependence of supply and demand on the price \cite{Goecke-play,laura}.
This model was fitted to microeconomic data based on a survey of German beer exports. It replaces
the demand and supply curves by play operators and predicts well the observed price rigidity.
The Preisach operator has been applied to modeling hysteresis in
unemployment \cite{scihys}.
Furthermore, the phenomenology of these hysteresis models is compatible with the multi-agent
modeling framework typical for economic models \cite{Lam1,Lamb,lambaimpl}.

\begin{figure}[h!]
	\begin{subfigure}{0.32\textwidth}
		\includegraphics[width=\textwidth]{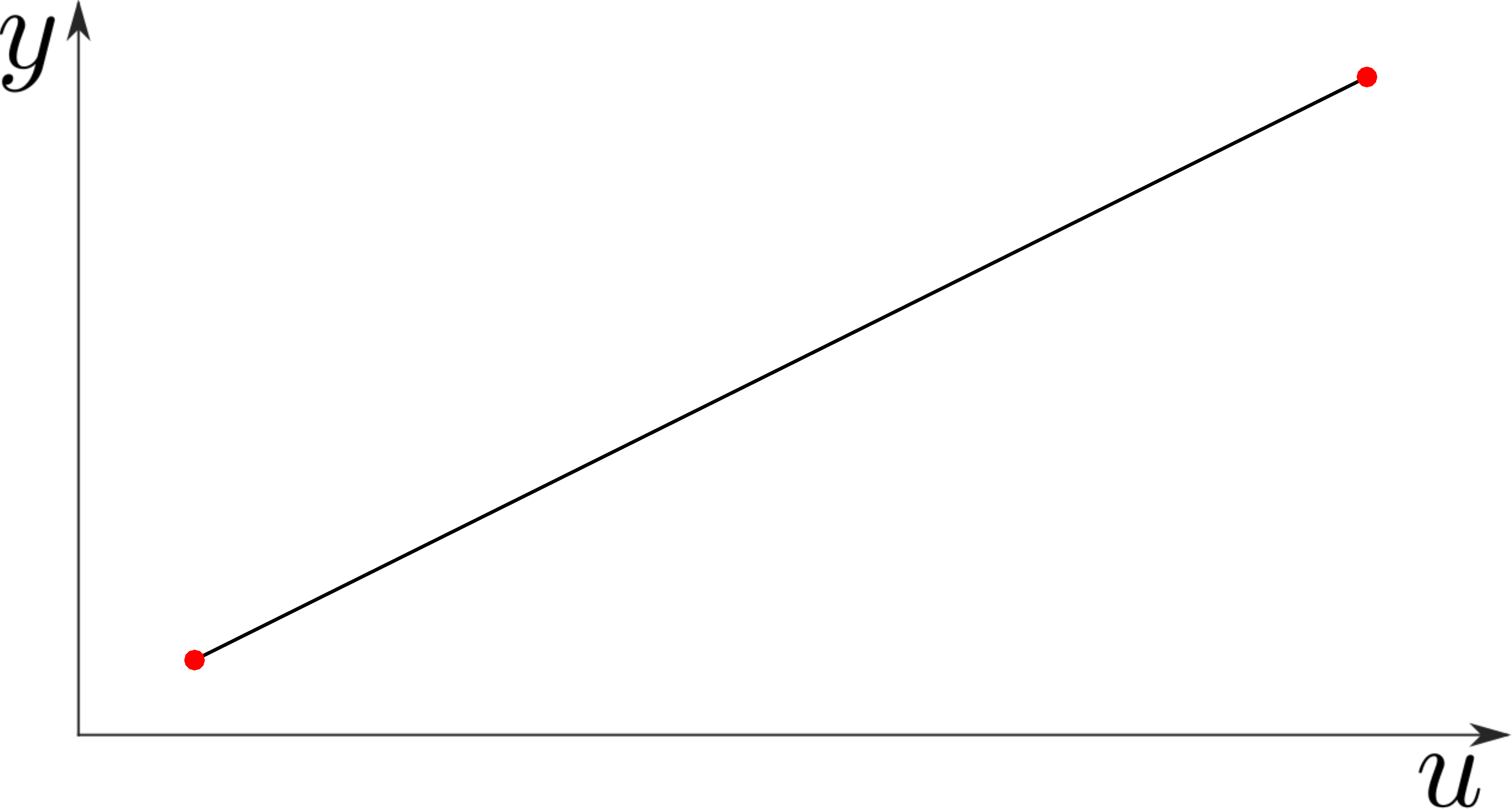}
		\caption{}
		\label{fig:NR2}
	\end{subfigure}
	\quad
	\begin{subfigure}{0.32\textwidth}
		\includegraphics[width=\textwidth]{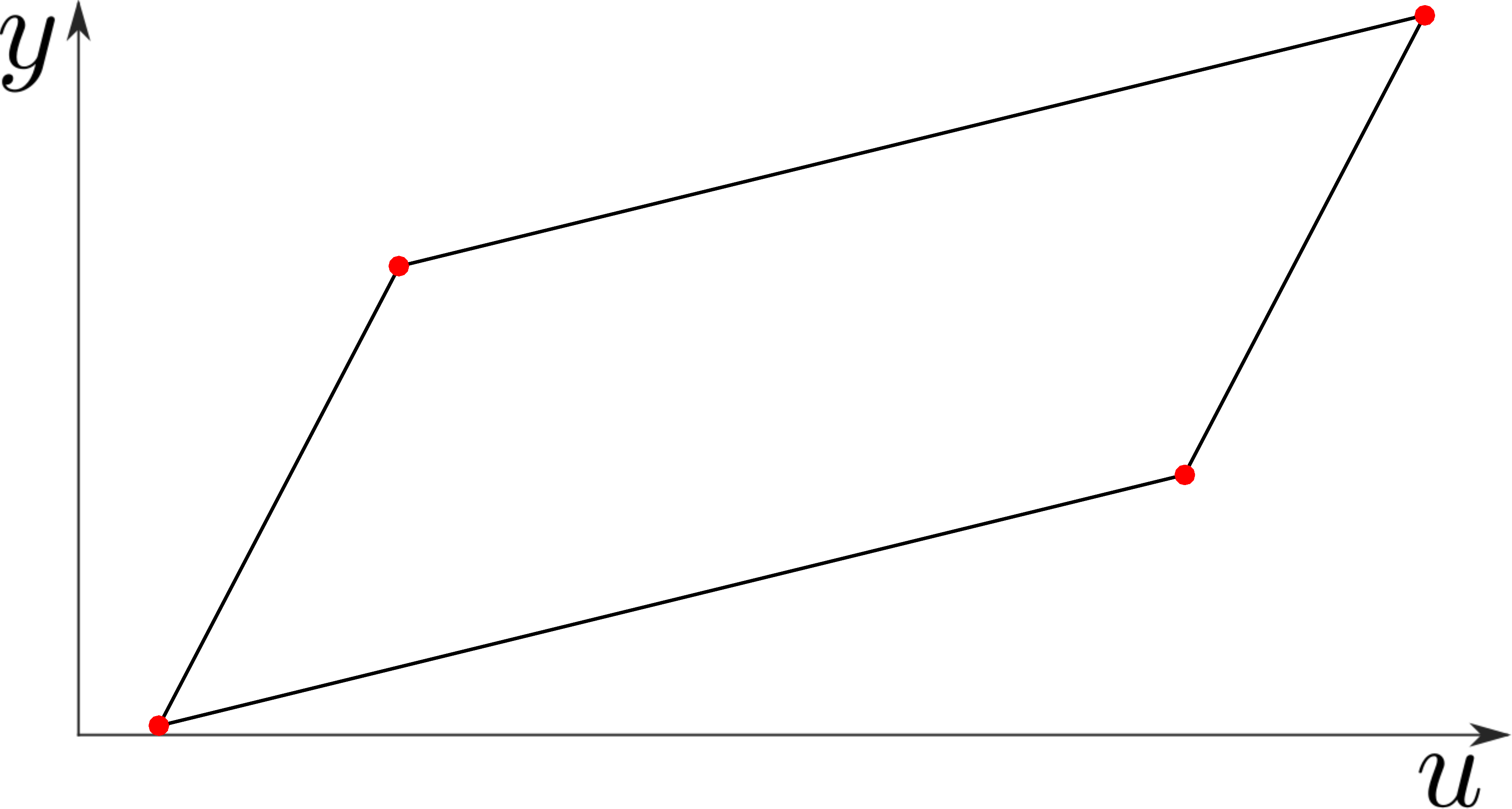}
		\caption{}\label{fig:NR3}
	\end{subfigure}
	\begin{subfigure}{0.32\textwidth}
		\includegraphics[width=\textwidth]{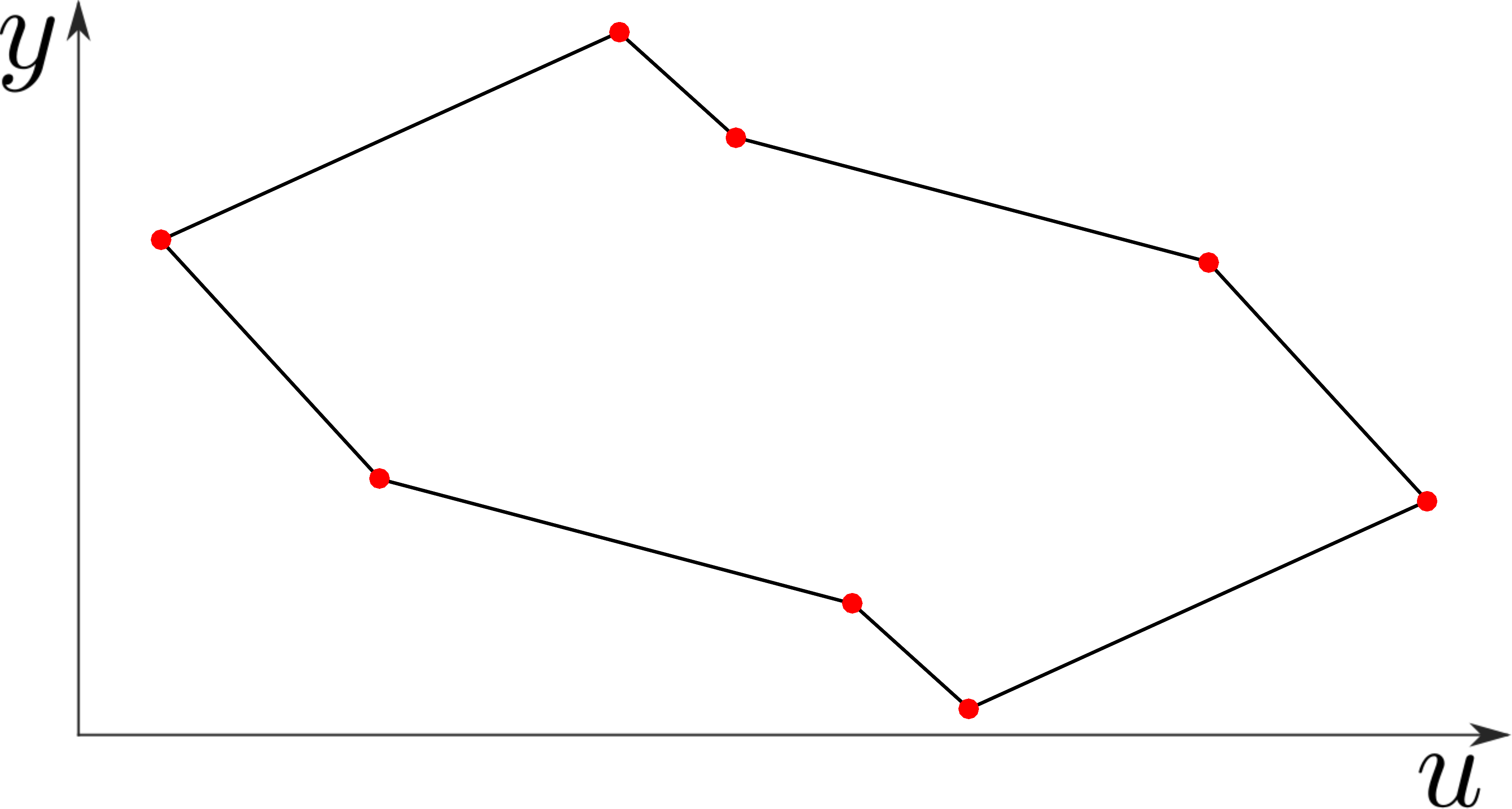}
		\caption{}\label{fig:NR4}
	\end{subfigure}
	
	\begin{subfigure}{0.45\textwidth}
		\includegraphics[width=\textwidth]{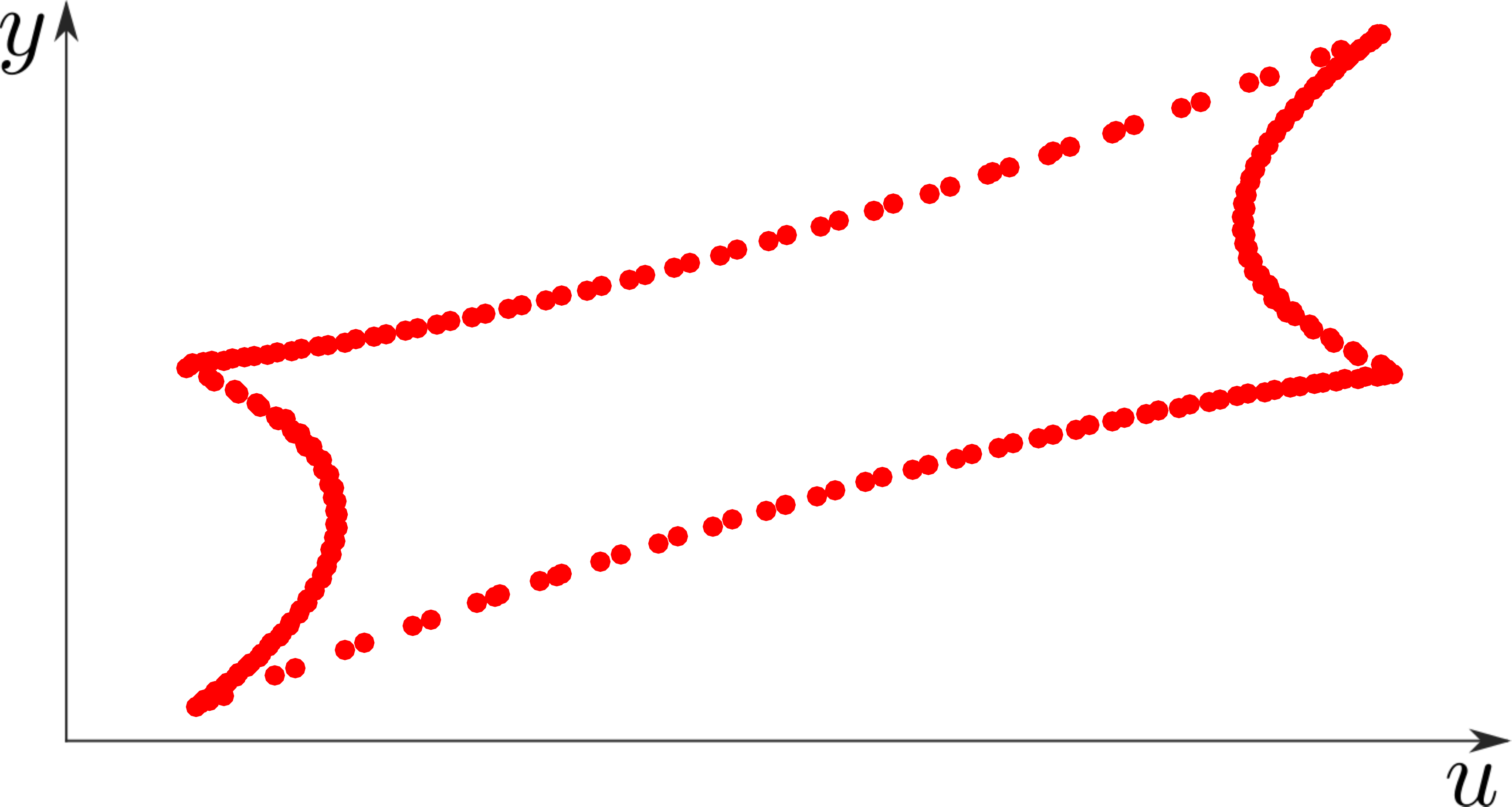}
		\caption{}\label{fig:NR5}
	\end{subfigure}
	\quad
\begin{subfigure}{0.45\textwidth}
\includegraphics[width=\textwidth]{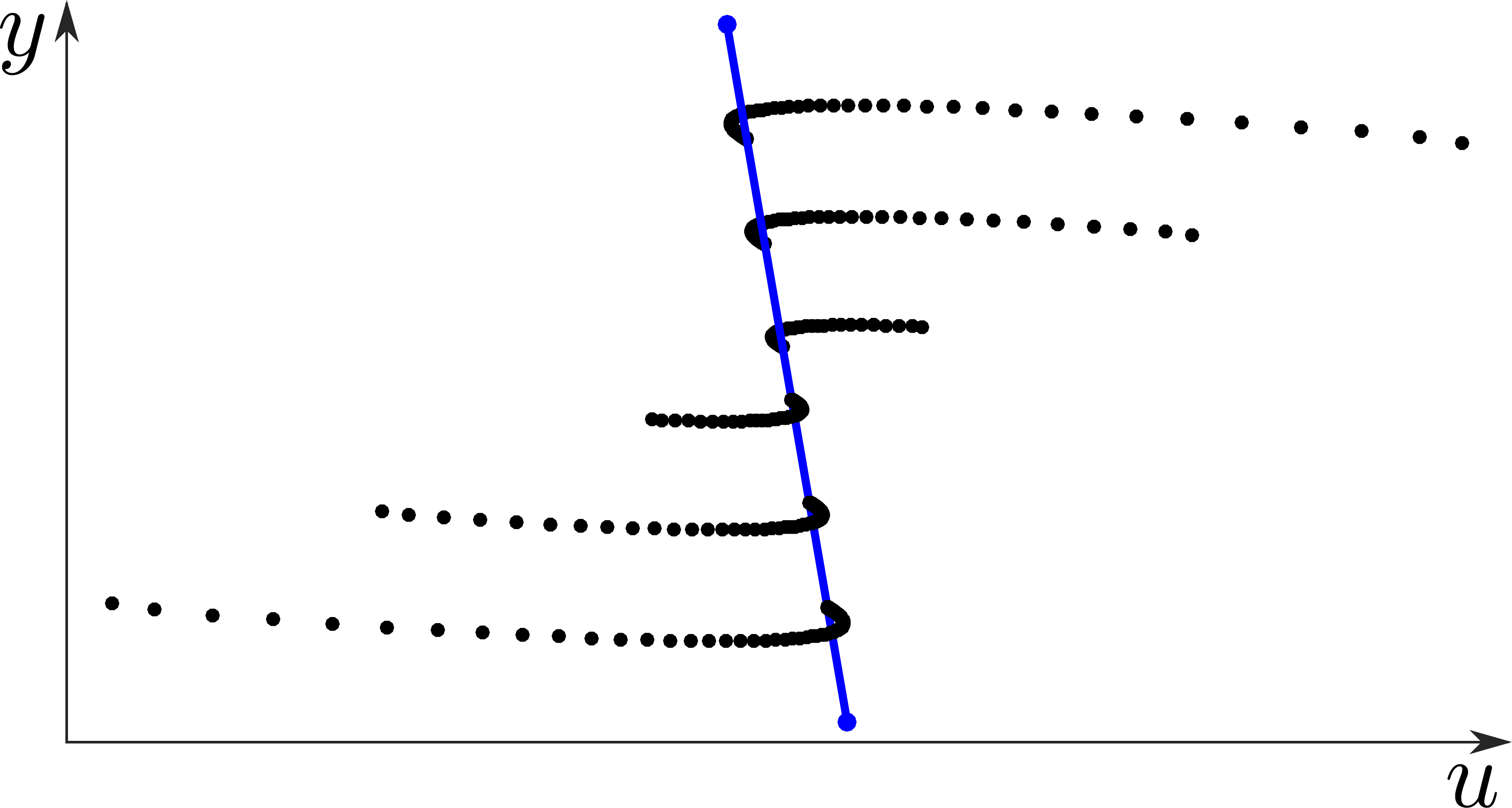}
\caption{}\label{fig:NR1}
\end{subfigure}
	\caption{Projections of various trajectories of system
	\eqref{11}--\eqref{play} onto $(u,y)$ plane. Parameter $\rho$ is fixed
	to be $1$, all the other parameters are given in the format  $a_1$,
	$b=(b_1,b_2)$, $c=(c_1,c_2,c_3)$.
	(a): $2$-periodic orbit  for
	$a_1=0.99$, $b=(0.76, 0.9)$, $c=(1.4, 9.7, 0.025)$.
	(b): $4$-periodic orbit for
	$a_1=0.7$, $b=(0.75, 0.5)$, $c=(4.8, 3.6, 3.45)$.
	(c): $8$-periodic orbit
	for
	$a_1=0.9$, $b=(0.73, 0.9)$, $c=(1.2,3.15, 1.3)$.
	(d): Quasi-periodic orbit
	for
	$a_1=0.7$, $b=(0.7, 0.55)$, $c=(4.8,
	4.15,
	3.8)$. (e): Trajectories converging to
	different points of  the segment of equilibrium points for $a_1=0.01$,
	$b=(0.01, 0.03)$, $c=(1, 6, 0.54)$.}\label{fig:NR6}
\end{figure}

The next natural step towards modeling the above effects in economics
would consist in formulation and analysis of closed models. With this motivation, let us consider
the following system, which belongs to the class of the so-called Dynamic
Stochastic General Equilibrium Models (DSGE) of macroeconomics (see, e.g.,
\cite{dsge6,dsge4,
dsge1,
dsge2}):
\begin{equation}\label{11}
\begin{cases}
y_{n+1}=
y_{n} - a_{1}(v_{n+1}-\sigma_{n+1})+\varepsilon_{n}, 
\\
u_{n+1}=b_1 \sigma_{n+1} +(1-b_1) u_{n} +b_2 y_{n+1}+\eta_{n},\\
v_{n+1}=c_1 (u_{n+1}-u^*)+c_2 y_{n+1}+c_3 v_{n}+\xi_n,%
\end{cases}
\end{equation}
where $y_n$ is the output gap (or employment rate, or another measure of activity of the economy), $u_n$ is the rate of inflation, $v_n$ is the 
interest rate, $\sigma_n$ is the aggregate of the economic agents' expectation of the future inflation rate, and
$\varepsilon_n$, $\eta_n$, $\xi_n$ are exogenous noise terms, see \cite{degrauwe}.
All the parameters are non-negative, 
$b_1<1$, and the parameter $u^*$, the inflation target, is for convenience set to zero.

In order to close the model, we need to complement system (\ref{11}) with an equation
defining how the economic agents' expectation  of the future inflation rate $\sigma_n$ is related to the actual inflation rate $u_n$.
We assume that $\sigma_n$ is related to $u_n$ via the
{\em play} operator:
\begin{equation}\label{play}
\sigma_{n+1}=u_{n+1} -\rho s_{n+1},\qquad s_{n+1}=\Phi\left(s_{n}+\rho^{-1}(u_{n+1}-u_{n})\right),
\end{equation}
where $\Phi$ is function \eqref{eq1}, see Fig. \ref{fig:0d}. Note that the
sequence $s_n$ is the
output of the {\em stop}
operator with the input $\rho^{-1} u_n$, see Fig. \ref{fig:0c}.
The play operator introduces inertia, or stickiness (with the associated memory), into the response
 of aggregated agent's expectation of inflation to variations of the rate of inflation $u_n$.
The parameter $\rho>0$ measures the maximum deviation of the expected rate of inflation from the actual rate.

Let us consider the unperturbed system \eqref{11}, \eqref{play}, i.e., we set the noise terms  $\varepsilon_n$, $\eta_n$, $\xi_n$ to zero.
This autonomous system can be rewritten in the explicit form
\begin{equation}\label{aaa}
z_{n+1}=\Lambda z_{n} + A s_{n},\qquad s_{n+1}=\Phi\left(s_{n}+\rho^{-1}(u_{n+1}-u_{n})\right),
\end{equation}
where $z$ is the column vector $z=(y,u,v)^T$, $\Lambda$ is a $3\times 3$ matrix, and $A\in\mathbb{R}^3$.
Therefore, system \eqref{eq2} can be viewed as a simpler one-dimensional counterpart of system \eqref{11}, \eqref{play}.

Fig. \ref{fig:NR6} presents various attractors of model \eqref{11},
\eqref{play} obtained numerically for different parameter regimes. In
particular, trajectories can converge to stable equilibrium points that form a
segment in the phase space, see Fig.~\ref{fig:NR1}. Alternatively, one can
observe convergence to a $2$-periodic orbit, or to a periodic orbit of a
higher
period, which coexists with the set of unstable equilibrium points (see
Figs.~\ref{fig:NR2}--\ref{fig:NR4}). Fig.~\ref{fig:NR5} indicates a possibility
of quasiperiodic dynamics. Comparing these scenarios with different cases
of Theorem \ref{Theorem:1} suggests that the prototype model \eqref{eq2}
can help understand some features of dynamics of the more complex
macroeconomic model \eqref{11}, \eqref{play}. Analysis of the latter model
is beyond the scope of this paper and will be the subject of future work.

%




\section{Proofs}


We will prove statements of Theorem \ref{Theorem:1} in the
counter-clockwise order along the bifurcation diagram in Fig.~\ref{fig:2b}.
Thus, we prove case \ref{item_b} in Section \ref{Case:1}, then case
\ref{item_a} for non-negative $\lambda$ in Section \ref{Case:2a}. Proofs
for  cases~\ref{item_c} and \ref{item_d}
are presented in Sections \ref{Case:4} and Section \ref{Case:7}, respectively.
In Section \ref{Case:2b}, we present the proof of the
remaining part of
case \ref{item_a}  for negative $\lambda$. Proofs of case
\ref{item_e} and of Theorem \ref{Theorem:2} are presented in Section
 \ref{Case:8}. Finally, Sections \ref{Case:9} and \ref{Case:10} contain
 the proofs for critical cases \ref{item_f} and~\ref{item_g}, respectively.

{\color{black} We use the following notations: $A_x$ and $A_s$ will denote
the $x$
	and $s$ coordinates of a point $A$ in the $(x,s)$-plain.
	Transformation \eqref{eq2} will be
	denoted by $f$.  Throughout the proofs, we will use the variable
	$p=x-s$ (output of
	the play operator, see Fig. \ref{fig:0d}).
	We will denote by $A_p=A_x-A_s$ the $p$-coordinate of a point $A$.}

{ \color{black} Let us start with a few preliminary remarks. First, due to the
fact that
\begin{equation}
\label{eq:f_odd}
f(-x,-s)=-f(x,s),
\end{equation} it is sufficient to present the proofs for
a half of the phase space.

 \begin{lemma}
 	\label{lm:mon_EF}
 	{\color{black} For any point $A$ to the left of the
 		segment $EF$, one has $[f(A)]_x>A_x$. For any point
 		$B$ to the right of the segment $EF$, one has $[f(B)]_x<B_x$.}
 \end{lemma}
 \begin{proof}{\color{black}
 	Since  $A$ lies to the left of the segment $EF$, one has
 	$(1-\lambda) A_x<a A_s$. Thus $[f(A)]_x=\lambda A_x+a A_s>\lambda
 	A_x+(1-\lambda)A_x=A_x$. The second statement follows from
 	\eqref{eq:f_odd}.}
 \end{proof}

\begin{lemma}
	\label{lm:mon_bet}
	{\color{black} Let $\beta>0$. Then for any two points $A$ and $B$ with
	the
	same $p$-coordinate $A_p=B_p$, from $A_x>B_x$ it follows
	$[f(A)]_x>[f(B)]_x$.}
	\end{lemma}
\begin{proof}
	{ \color{black}  It suffices to note that
		$
		[f(A)]_x=\lambda A_x+a A_s= \beta A_x-aA_p.
		$}
\end{proof}


Denote by $\Pi\subset L$ the parallelogram
with the diagonal $EF$, two sides on the lines $s=\pm1$, and two sides
with slope $1$:
\begin{equation}\label{PI}
\Pi=\left\{(x,s):\ \left| x-s\right|\le \left|\frac{a}{1-\lambda}-1\right|,\ |s|\le 1 \right\}.
\end{equation}

\begin{lemma}
	\label{lm:Pi_par}
For $0\le \lambda<1$, $\beta\le 0$ and for $-1<\lambda\le 0$, $-1\le \beta\le 0$ the
parallelogram
	$\Pi$ is invariant under the map
	\eqref{eq2}.
\end{lemma}
\begin{proof}
	Let $(x_n,s_n)\in\Pi$, i.e. $|p_n|\le\frac{1-\beta}{1-\lambda}$.
Since the upper-right and the lower-left vertices of $\Pi$ are the points $E'=\left(-\frac{a}{1-\lambda}+2,1\right)$ and
$F'=\left(\frac{a}{1-\lambda}-2,-1\right)$, respectively, it suffices to prove that
\begin{equation}\label{***}
|x_{n+1}|\le -\frac{a}{1-\lambda}+2.
\end{equation}
If $0\le \lambda<1$, $\beta\le 0$, then
\[
|x_{n+1}|\le \lambda \,\frac{1-\beta}{1-\lambda}-\beta=-\frac{a}{1-\lambda},
\]
which implies \eqref{***}. If $-1<\lambda\le 0$, $-1\le \beta\le 0$, then
\[
|x_{n+1}|\le -\lambda \,\frac{1-\beta}{1-\lambda}-\beta,
\]
which yields \eqref{***} because $\lambda<1$ and $\lambda\beta\le 1$.
\end{proof}

\subsection{Case \ref{item_b}}\label{Case:1}
In this case, $\lambda\ge 0$ and $\beta\geq1$. Therefore $a>0$ and
the slope of the segment $EF$ of equilibrium points is positive and less
than or equal to $1$ as shown in Fig.~\ref{fig:2d}.

\begin{figure}[hbt]
	\begin{center}
		\includegraphics[width=0.48\textwidth]{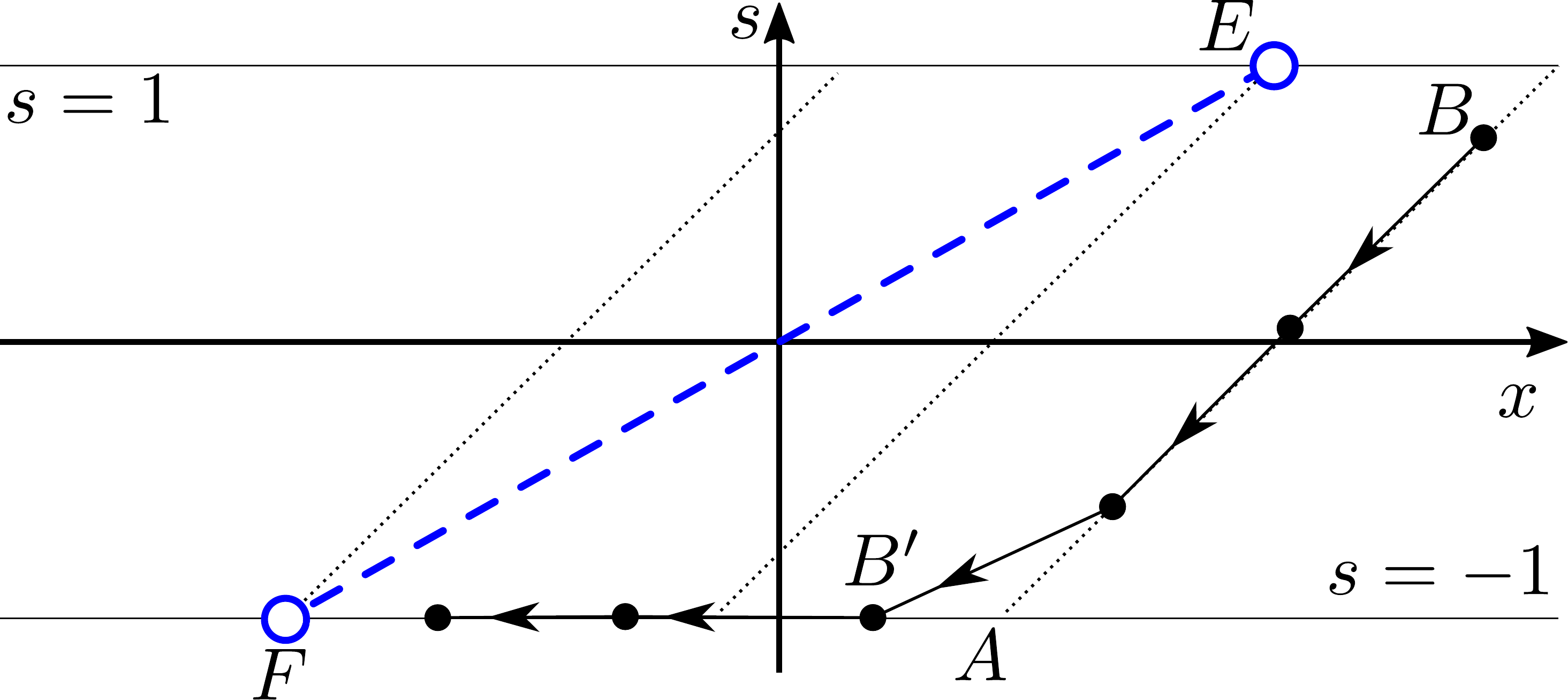}
		\caption{$\lambda\ge 0$, $\beta \ge 1$. A trajectory starting at
			a
			point $B$
			to
			the right of the
			segment $EF$ converges to the
			fixed point $F$. Dotted lines have slope $1$.}  \label{fig:2d}
	\end{center}
\end{figure}

Consider a point $B$ which lies to the right of the segment
$EF$. Denote by $A$ the intersection point of the lines $p=B_p$ and
$s=-1$. Let $B'$ denote the point at which the trajectory $\{f^n(B)\}$
hits
the line $s=-1$ for the first time.
From Lemma \ref{lm:mon_EF} it follows that $[f^{-1}(B')]_x>A_x$, and
$[f(A)]_x>F_x$ since $\lambda>0$. Thus, from Lemma
\ref{lm:mon_bet} we obtain $B'_x>F_x$.
Since $B'_s=-1$, it follows that $[f(B')]_x-F_x=\lambda (B'_x-F_x)$. Hence,
due to $\lambda\in [0,1)$, the trajectory converges to the equilibrium $F$
along
the line $s=-1$ (see Fig.~\ref{fig:2d}). We conclude that every trajectory
that
starts to
the right of the segment $EF$ of  equilibrium points, converges to $F$.
Every trajectory which starts to the left of $EF$ converges to $E$ due to
\eqref{eq:f_odd}.

\subsection{Case \ref{item_a}, $\lambda\ge 0$}\label{Case:2a}
\subsubsection{$0<\beta<1$}\label{Case:2}

In this case,
the segment $EF$ has a positive slope greater than $1$ if $a>0$
and nonpositive if $a\leq0$.

\paragraph*{1}
 First, let us consider  the trajectory of a point $A$ that belongs to the parallelogram $\Pi$ defined by \eqref{PI}, see Fig.~\ref{fig:b+}. Denote by $P^*$ the point of intersection of the line
$p=A_p$ with the
segment
$EF$ of equilibrium points.
Since
$$ [f(A)]_p=A_p,\quad \quad [f(A)]_x-P_x^*=\beta(A_x-P_x^*)$$
with $\beta\in(0,1)$, the trajectory of $A$ converges to the point $P^*$.

\paragraph*{2} Thanks to Lemma \ref{lm:mon_bet}, all the other
trajectories that start to
the right of the parallelogram $\Pi$, move down along the line
$p=const$ until they hit the line $s=-1$ and then monotonically
converge to the equilibrium point $F$ along this line from the
right, see Fig.~\ref{fig:2e}.

\begin{figure}[h]
\centering
\begin{subfigure}[t]{0.48\textwidth}
\includegraphics[width=\textwidth]{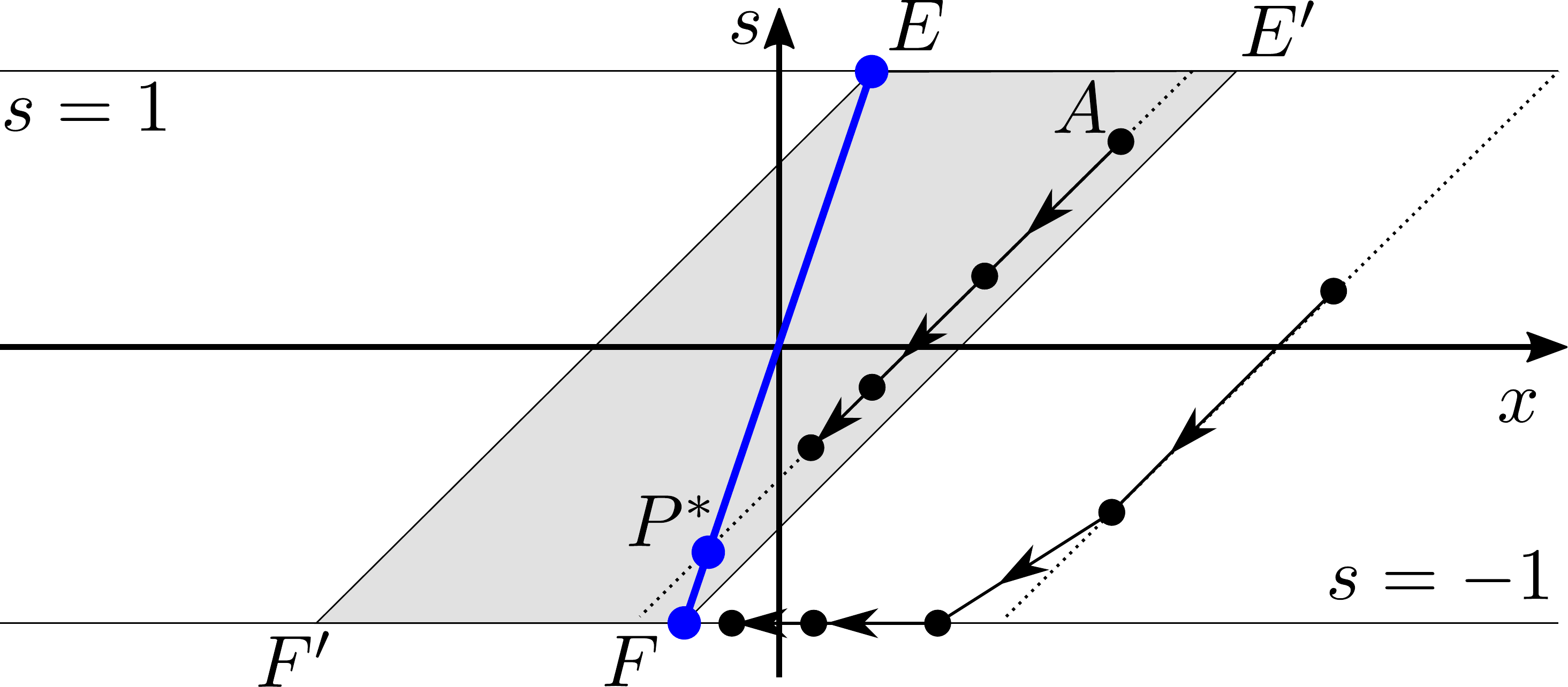}
\caption{$0<\beta<1$.
	}
\label{fig:2e}
\end{subfigure}
\begin{subfigure}[t]{0.48\textwidth}
	\includegraphics[width=\textwidth]{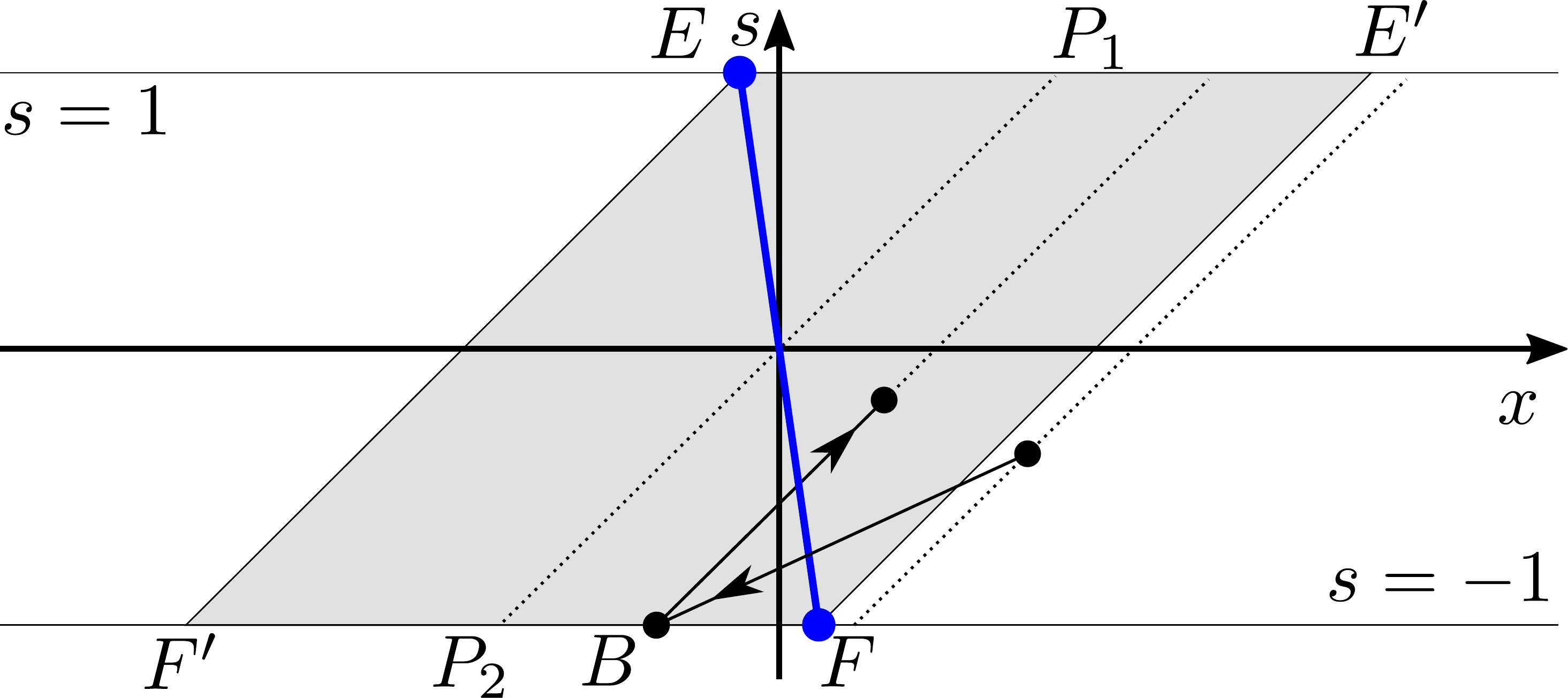}
\caption{$-1<\beta\leq0$.}
\label{fig:2f}
\end{subfigure}
\caption{Theorem \ref{Theorem:1}\ref{item_a}. Case $\lambda \ge0$.
Dotted lines
have slope $1$. The shaded area is the parallelogram $\Pi=EE'FF'$.}
\label{fig:b+}
\end{figure}

%

\subsubsection{ $-1<\beta\le
0$}\label{Case:3}
In this case, $a<0$ and the
segment $EF$ has a negative slope as in Fig.~\ref{fig:2f}.

\paragraph*{1} If a trajectory
starts to the right of the parallelogram $\Pi$, then, since $\beta\leq 0$, it
hits the line $s=-1$ after one iteration. If it hits the line to the right of the
equilibrium
$F$, then the trajectory converges to this equilibrium along the line $s=-1$
from the right due to $\lambda\ge0$.
On the other hand, if this trajectory hits the line $s=-1$ at a point
$B$ to the left of the point $F$, then $B$ belongs to the
parallelogram $\Pi$. In order to show this, we note that for the 
previous point
$f^{-1}(B)=(x_n,s_n)$, we have
\[
x_n-s_n>-\frac{a}{1-\lambda}+1,
\]
because the point $(x_n,s_n)$  lies to the right of the parallelogram $\Pi$.
Therefore,
\begin{eqnarray}
x_{n+1}&=&\lambda x_n+ as_n 
> \lambda\left(-\frac{a}{1-\lambda}+1+s_n\right)+as_n\nonumber\\
&\ge &\lambda\left(-\frac{a}{1-\lambda}+1\right)+\lambda+a. \nonumber
\end{eqnarray}
This last expression is 
greater than
$F'_x= \frac{a}{1-\lambda}-2$, which is the $x$-coordinate
of the lower left vertex of the parallelogram $\Pi$.

\paragraph*{2} Consider  points on the horizontal sides of $\Pi$. To
be definite,
assume that $s_n=-1$. Denote by $P_{1}=(1,1)$ and $P_{2}=(-1,-1)$ the
middle points of $EE'$ and $FF'$, respectively. If $(x_n,s_n)\in P_{2}F$, then
the trajectory converges to the equilibrium along the line $p=const$. Let
$(x_n, s_n)\in F'P_2$. If $s_{n+1}<1$, then the
trajectory converges to the equilibrium along the line $p=const$. If $s_{n+1}=1$, then $(x_{n+1},s_{n+1})\in EP_{1}$ since $x_{n+1}=\lambda x_{n}-a\leq-\lambda-a<1$. Again, the trajectory converges to the equilibrium along the line $p=const$.

%

It remains to consider points in $\Pi$ that belong to the open band $|s|<1$. A trajectory starting from such a point either converges to an equilibrium along the line $p=const$ without hitting the lines $s=\pm1$, or hits one of these lines and then converges to an equilibrium as discussed above.


\subsection{Case \ref{item_c}}\label{Case:4}

As in the previous case,
$a<0$ and the segment $EF$ has a negative slope (see Fig.~\ref{fig:2g}). In
this case, there is a $2$-periodic orbit. A $2$-periodic orbit consists of the
points
$\pm Q$ where $Q=(Q_x,1)$.
Here, $Q_x=-\lambda Q_x-a$, therefore
\[Q_x=-\frac{a}{1+\lambda}.\]
Since $\beta<-1$, it follows that $Q_x>1$, so the distance between the
$x$-components of $Q$
and $-Q$ is larger than $2$ and thus periodic points indeed belong to the
lines
$s=\pm 1$.

\paragraph*{1} Now we consider dynamics of different trajectories. Denote
\begin{equation}
\label{eq:AB}
A=\left(\frac{a+2}{1-\lambda},1\right), \quad
B=\left(\frac{1}{\lambda}\left(-\frac{a+2}{1-\lambda}-a\right),1\right).
\end{equation}
 If $\lambda=0$, we formally set
$B_x=\infty$ and replace the segment $AB$ below  by the corresponding
half-line. Note that $B_x>A_x$ and $Q\in AB$.

\begin{lemma}
	\label{lm:AB}
	 Under the hypothesis of Theorem \ref{Theorem:1} \ref{item_c},
	the segment  $AB$ is
invariant under the second iteration $f^2$ of the map $f$ and $f^2$ is a
contraction on $AB$.
\end{lemma}

\begin{proof}
	First, we note that if a point $(x_n,s_n)$ lies on the line $s=1$ to the
	right of the point~$A$, then the image $(x_{n+1},s_{n+1})=f(x_n, s_n)$ of
	this
	point under the map \eqref{eq2} belongs to the line $s=-1$.
Indeed,
\[x_{n+1}-x_n=\lambda x_n+a-x_n\leq(\lambda-1)A_x +a=-2,\]
which implies $s_{n+1}=-1$. By \eqref{eq:f_odd}, the points of the line
$s=-1$ lying to the left of $-A$
are mapped to the line $s=1$.
Furthermore, if a point $(x_n,s_n)$ lies on the line $s=1$ between the points $A$ and $B$,
then its image $f(x_n,s_n)=(\lambda x_n+a,s_{n+1})$ lies on the line $s=-1$ to the left of the point $-A$, and therefore the second iteration $f^{2}(x_n,s_n)$ belongs to the line $s=1$.
Hence, the segment $AB$ is mapped by the second iteration $f^{2}$ to the line $s=1$.
If $(x_n,1)\in AB$, then $f^{2}(x_n,1)=(\lambda^2x_n+\lambda a-a,1)$. In
particular, $[f^{2}(A)]_x$ is defined by the expression
\[
\lambda^2A_x+\lambda a-a
=\frac{-a+2\lambda^2+2a\lambda}{1-\lambda}.
\]
The inequality
\[
\frac{-a+2\lambda^2+2a\lambda}{1-\lambda}>\frac{a+2}{1-\lambda},
\]
which is equivalent to
$(\lambda-1)(\beta+1)>0,$
ensures that $[f^{2}(A)]_x\geq A_x$.

To show that $[f^{2}(B)]_x\leq B_x$, it is sufficient to check the following
inequality:
\[
\lambda^2\left(\frac{-2a-2+a\lambda}{\lambda(1-\lambda)}\right)+
a(\lambda-1)<\frac{-2a-2+a\lambda}{\lambda(1-\lambda)}.
\]
After a simple manipulation, this inequality follows from
$\beta+1<0$. Since $f^{2}$ maps $A$ and $B$ into $AB$ and
$[f^{2}(x,1)]_x$ is increasing on $AB$ with respect to $x$, we
conclude that
$f^{2}(AB)\subseteq AB$. Since $\lambda^2<1$, we see that
$f^{2}$ is a contraction on $AB$, hence the trajectories starting in $AB$
converge to the fixed point $Q=(\frac{-a}{1+\lambda},1)$ of $f^2$, which belongs to
the segment $AB$.
\end{proof}
\begin{figure}[h]
\begin{center}
\includegraphics[width=0.5\textwidth]{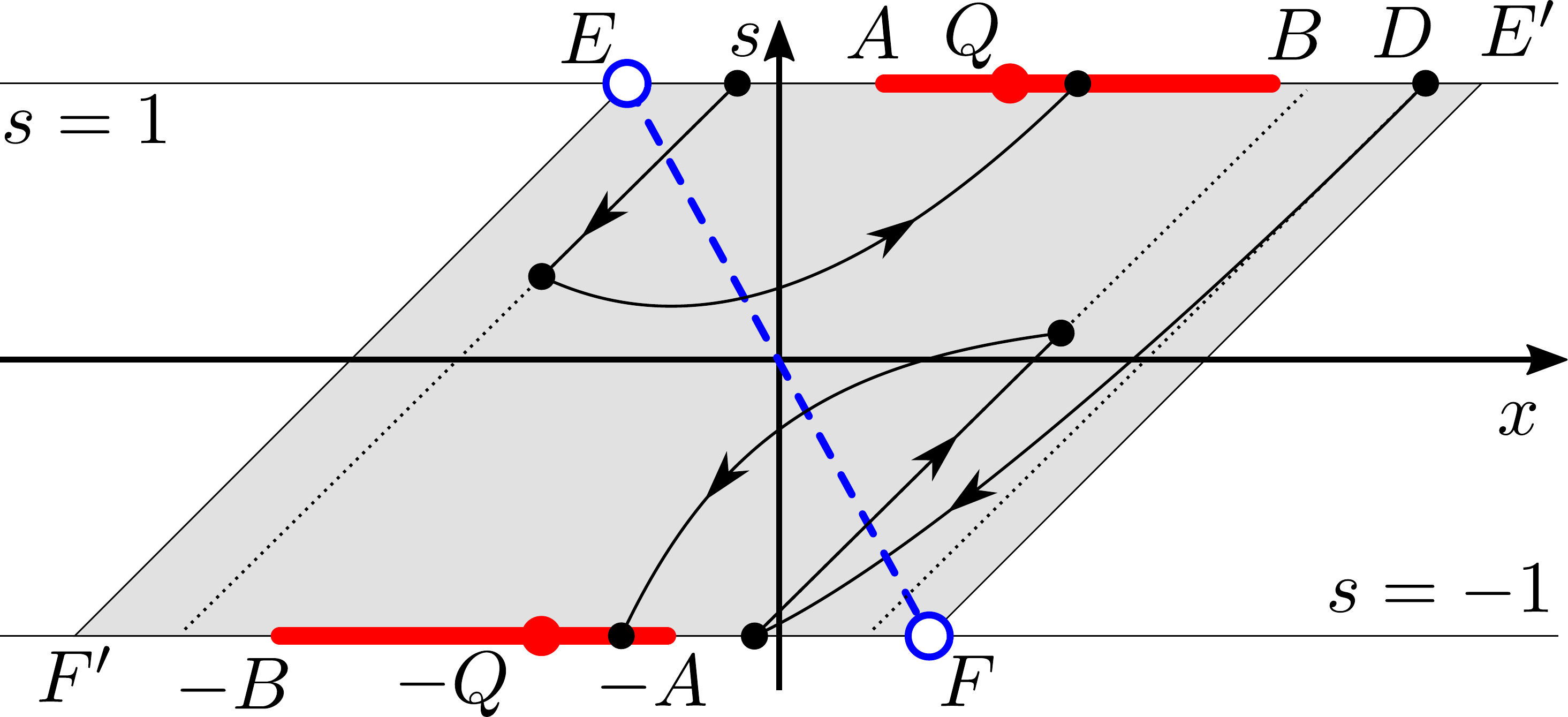}
\caption{Case $\lambda\geq 0$, $\beta<-1$. Each of the red segments
$AB$ and $-AB$ is mapped into
itself by $f^2$.
}\label{fig:2g}
\end{center}
\end{figure}

\paragraph{2}\label{par:4.3.2} Next, we consider the situations
where
$B$
lies to the right
of $E'$ and
where $B$ lies between $A$ and $E'$, respectively. In the former case, any
trajectory
starting between $A$ and $E'$ converges to the $2$-periodic orbit due to
the
above argument. Consider the latter case. Let a trajectory start on the upper
side of the parallelogram $\Pi$ to the right of the point $B$ at a point
$D=(x_n,1)$, see Fig.~\ref{fig:2g}. The image $f(D)=(x_{n+1},-1)$ of this
point lie on the line $s=-1$ to the right of the point $-A$. Therefore,
$f^{2}(D)$ belong to the interior of the strip $L$. Since $\beta<-1$,
further iterations $f^{n+k}(D)$ belong to the line $s=-1$ for odd $k$ and
to the interior of $L$ for even $k$, and the $x$-coordinate of the odd iterations monotonically decreases until the trajectory reaches
 the half-line $\{(x,s):\, x\le -A_x,\, s=-1\}$. Without loss of generality, we can assume that $f(D)$ is the last point of the trajectory, which is still to the right of the point $-A$ on the line $s=-1$. Let us show that the point $(x_{n+3},-1)$  lies to the right of the point $(\frac{a}{1+\lambda},-1)$.
To this end, we note that
\[
\begin{array}{c}
(x_{n+2},s_{n+2})=(\lambda x_{n+1}-a,-1+\lambda x_{n+1}-a-x_{n+1}),
\\
x_{n+3}=\lambda^2 x_{n+1}-\lambda a +a(-1+\lambda x_{n+1}-a-x_{n+1}).
\end{array}
\]
Thus, we need to show that
$x_{n+1}>- \frac{a+2}{1-\lambda}$
implies
\begin{equation}\label{eqnCase5555}
	\lambda^2 x_{n+1}-\lambda a+a(-1+\lambda x_{n+1}-a-x_{n+1})>\frac{a}{1+\lambda},
\end{equation}
i.e.,
\[(\lambda^2+a\lambda-a)x_{n+1}-a(\beta+1)>\frac{a}{1+\lambda}.\]
Since $\lambda^2+a\lambda-a>1$, it suffies
to show \eqref{eqnCase5555} for
$x_{n+1}=-\frac{a+2}{1-\lambda},$
i.e.,
\[(\lambda^2+a\lambda-a)\left(-\frac{a+2}{1-\lambda}\right)-a(\beta+1)>\frac{a}{1+\lambda}.\]
But this is equivalent to
\[\lambda^2(\beta+1)<0,
\]
which is true in the case we are considering.
  We see that the point $(x_{n+3},-1)$ belongs to the segment connecting
  the points $-A$ and $-B$, which is invariant for the map $f^{2}$ thanks
  to Lemma \ref{lm:AB}. Hence
  the trajectory converges to the $2$-periodic orbit.

\paragraph{3}
Next, we consider a trajectory which starts at a point $D'$ on the line $s=1$ to the left of the point $A$ in the parallelogram $\Pi$. For this trajectory, further odd iterations $f^{k}(D')$ lie in the interior of $L$, while the even iterations $f^{k}(D')$ belong to the
 line $s=1$, and the $x$-coordinate of the even iterations monotonically
 increases until the trajectory reaches the segment $AB$. (This behaviour is
 similar to the behaviour that we considered in paragraph $2$).
 Hence,
 such a trajectory also converges to the $2$-periodic orbit.

Any trajectory that starts in the parallelogram $\Pi$, but not on the lines
$s=\pm1$ and not on the segment of equilibrium points, thanks to
Lemma \ref{lm:Pi_par} will stay inside
$\Pi$. It reaches one of the lines $s=\pm1$ in several iterations due to the
condition $\beta<-1$. Thus, we see that all the trajectories that start
in the parallelogram $\Pi$ except for the segment of equilibrium points,
converge to the $2$-periodic orbit.

\paragraph*{4} Finally, let us consider a trajectory that starts to the right
of the
parallelogram $\Pi$. Since $\beta<0$,
this trajectory
 reaches the line $s=-1$ after one iteration. If it reaches this line to the
 right of the
equilibrium point $F$, then it will move to the left along the line $s=-1$
and converge to the equilibrium point $F$ from the right. On the other
hand, if a trajectory reaches the line $s=-1$ at a point, which lies to the
left of the point $F$, then this point belongs to $\Pi$. This can be shown
exactly in the same way as we did in Section~\ref{Case:3}. Therefore, such a
trajectory converges to the $2$-periodic orbit. We conclude that the
$2$-periodic orbit is stable and its basin of attraction contains the
parallelogram $\Pi$ with the exception of equilibrium points. However,
some trajectories from outside the parallelogram $\Pi$ are attracted to the
semi-stable equilibrium points $E$ and $F$.

\subsection{Case \ref{item_d}}\label{Case:7}
In this case, $a<0$ and the segment
$EF$ has a negative slope (see Fig.~\ref{fig:2h}).
 Like in Section \ref{Case:4}, there exists a $2$-periodic orbit $\pm Q$ defined by
 \eqref{per2}. Let $A$ be as in \eqref{eq:AB}.

First, we note that if a point $(x_n,-1)$ satisfies
$x_n\leq-A_x$,
then $x_{n+1}>\frac{a+2}{1-\lambda}$,  $s_{n+1}=1$.
 Hence the half-line $\{(x,s):\, x\le -A_x,\, s=-1\}$ is mapped to itself
 under $f^2$. Since $x_{n+2}=\lambda^2x_n-\lambda a+a$ and
 $\lambda^2<1$, any trajectory starting at this half-line converges to the
 $2$-periodic orbit $Q$.
\begin{figure}[h]
\begin{center}
\includegraphics[width=0.5\textwidth]{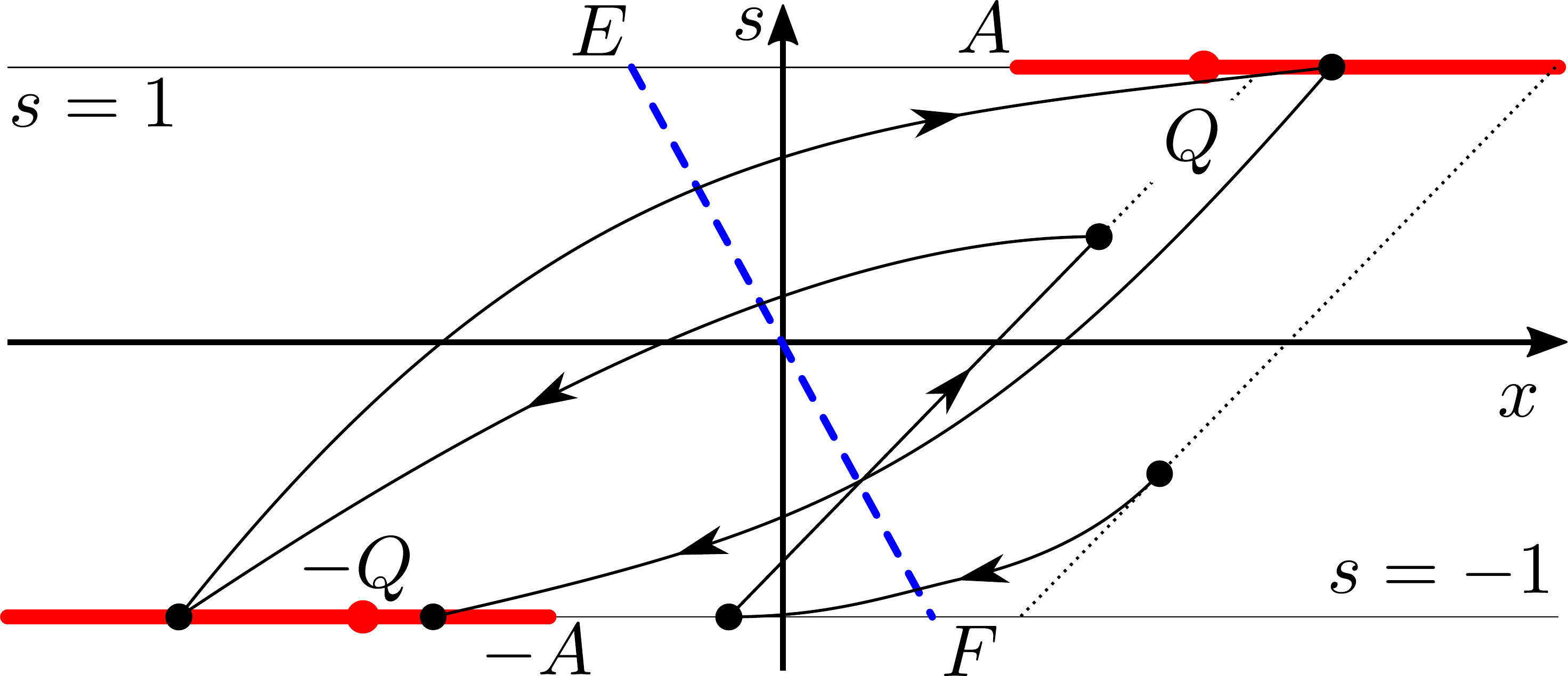}
\caption{Case $\lambda<0$, $\beta<-1$. Red half-lines are mapped to
themselves by $f^2$.}\label{fig:2h}
\end{center}
\end{figure}

If a point belongs to the open segment $\{(x,s):\, -A_x <x<
\frac{-a}{1-\lambda},\, s=-1\}$, then its trajectory enters the half-line
$\{(x,s):\, x\le -A_x,\, s=-1\}$ after finitely many iterations because
$\beta<-1$ (like in Section \ref{Case:4}). Hence, the half-line $\{(x,s):\, x<
\frac{-a}{1-\lambda},\, s=-1\}$ belongs to the basin of attraction of the
$2$-periodic orbit. If a point belongs to the half-line $\{(x,s):\, x>
\frac{-a}{1-\lambda},\, s=-1\}$, then its first iteration is in the half-line
$\{(x,s):\, x< \frac{-a}{1-\lambda},\, s=-1\}$ because $\lambda<0$.
 Hence, we conclude that the lines $s=-1$ and $s=1$ (except for the
 equilibria $F$ and $E$, respectively) belong to the basin of attraction of
 the $2$-periodic orbit.

 Finally, all trajectories that start inside the strip $-1<s<1$, except for the
 equilibrium points, will reach one of the lines $s=\pm1$ after finitely
 many iterations because $\beta<-1$.
Therefore, the $2$-periodic orbit attracts all the trajectories except for the
equilibrium points and their pre-images.

\subsection {Case \ref{item_a}, $\lambda<0$}\label{Case:2b}
\subsubsection{$-1<\beta<0$}\label{Case:6}

For the point $(x_n,s_n)$ to the right of the parallelogram $\Pi$,
one has
\[
x_n-s_n>p^*=1-x^*=1-\frac{a}{1-\lambda},
\]
and so
$x_{n+1}<\lambda p^*-\beta
=-x^*$,
hence the point $(x_{n+1},s_{n+1})$ lies to the left of the equilibrium $F$
on the line $s=-1$, see Fig.~\ref{fig:2i}.  Due to \eqref{eq:f_odd},  for
the point $(x_n,s_n)$ to the left of $\Pi$,
its image will lie on the line $s=1$ to the right of the point $E$.

Now, we prove that every trajectory enters $\Pi$. Arguing by contradiction,
let us show that if a trajectory
never entered $\Pi$, then
the distance from the trajectory to $\Pi$ would exponentially decrease.
This would imply that such a trajectory converges to a $2$-periodic orbit,
because, as we have seen,
its points belong to the union of the lines $s=\pm1$ and the sign of $s_n$
alternates at every iteration.
However, this is impossible as a $2$-periodic orbit does not exist in the
case
we are considering.

In order to see that the distance from a trajectory to $\Pi$ exponentially
decreases, it is sufficient to establish the inequality
\[
q\left(\frac{a}{1-\lambda}-2-x_n\right)> \lambda x_n -a
+\frac{a}{1-\lambda}-2
\]
for $x_n<\frac{a}{1-\lambda}-2$ and some $q\in (-\lambda,1)$
independent of $n$.
This inequality can be written as
\[x_n<\frac{a-(\frac{a}{1-\lambda}-2)(1-q)}{q+\lambda}.
\]
Thus we need to show that
\[
\frac{a}{1-\lambda}-2<\frac{a-(\frac{a}{1-\lambda}-2)(1-q)}{q+\lambda},
\]
which is equivalent to
$
(1+\lambda)\left(\frac{a}{1-\lambda}-2\right) < a
$
and, further, to $\beta\lambda<1.$
Since the last inequality is true in the case being considered, we can use any
$q\in(-\lambda,1)$.
The above argument shows that every trajectory enters the parallelogram
$\Pi$.

Since $\beta\in (-1,0)$ and $\lambda\in (-1,0)$, it follows from Lemma
\ref{lm:Pi_par} that $\Pi$
is invariant for the map $f$. Further, we note that if
some iteration of a point from $\Pi$ is mapped in the interior of $L$, then
the
trajectory converges to an equilibrium due to
$|\beta|<1$, see Fig.~\ref{fig:2i}.

%
%

\begin{figure}[h]
	\centering
	\begin{subfigure}[h!]{0.48\textwidth}
		\includegraphics[width=\textwidth]{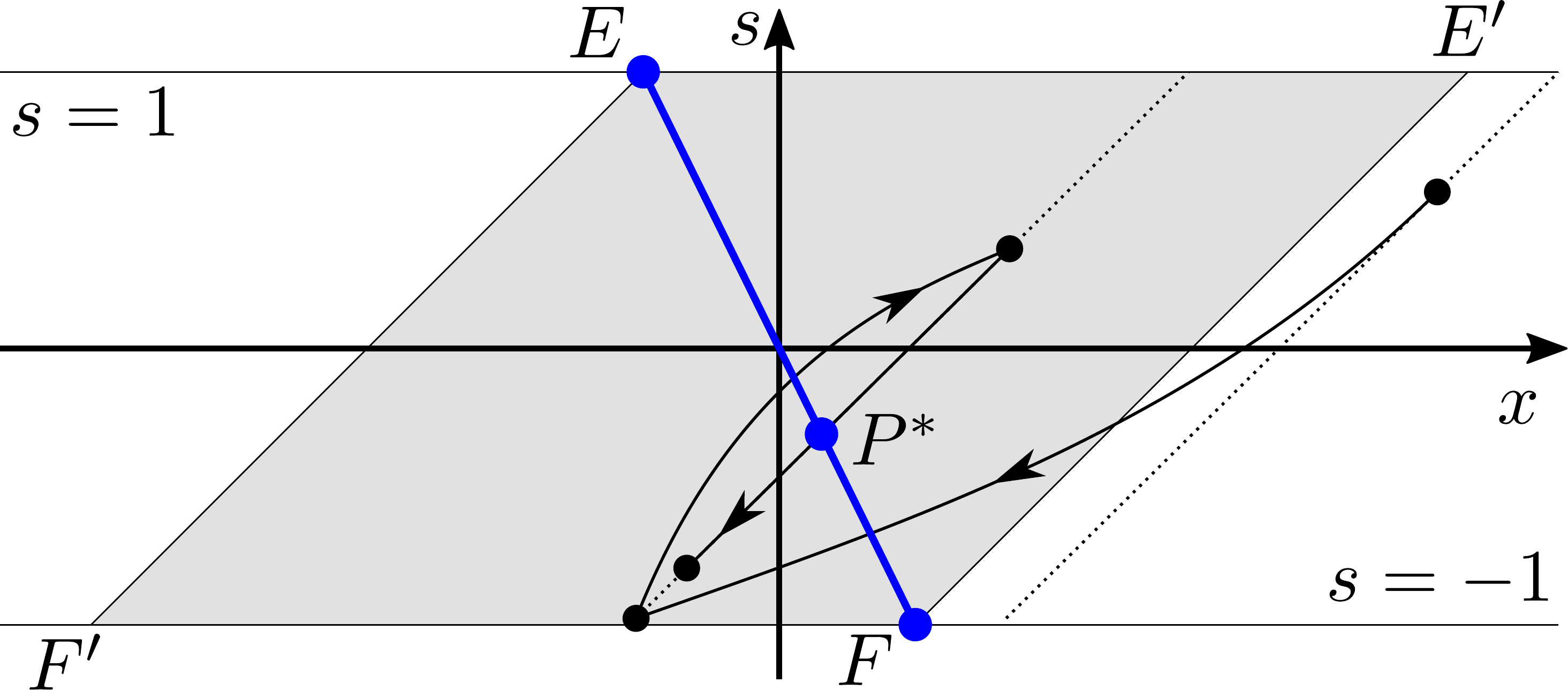}
		\caption{$-1<\beta<0$.}\label{fig:2i}
	\end{subfigure}
	\begin{subfigure}[h!]{0.48\textwidth}
		\includegraphics[width=\textwidth]{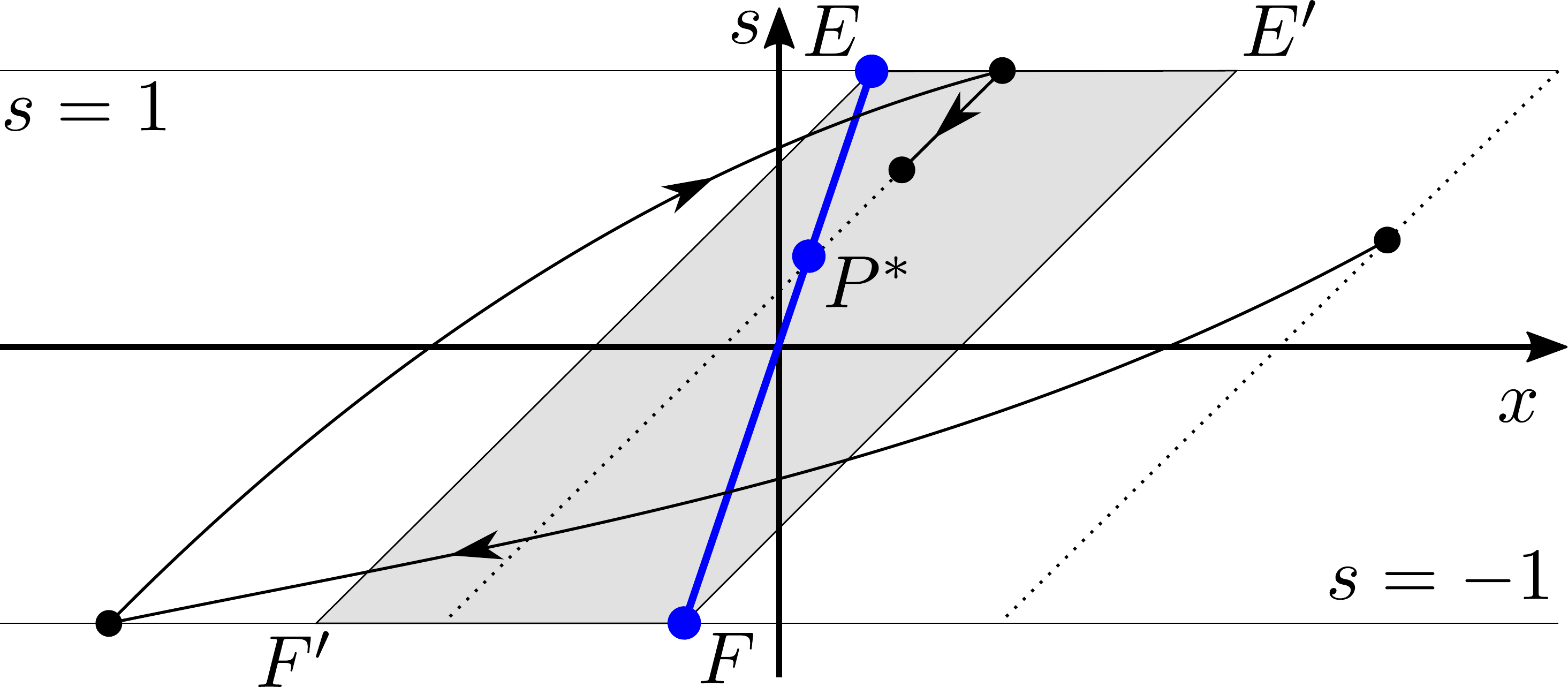}
		\caption{$0\leq\beta<1$.}\label{fig:2j}
	\end{subfigure}
	\caption{Theorem \ref{Theorem:1}\ref{item_a}. Case $\lambda<0$.}
\end{figure}

Finally, let us show that a trajectory cannot jump from the line $s=1$ to the
line $s=-1$ and back all the time. Indeed, if this was the case, then a point
$(x_n,1)$ from this trajectory would satisfy $x_n>1$ and the point
$(x_{n+1},-1)$ would satisfy $x_{n+1}=\lambda x_n+a<-1$. But
inequalities $-1<\beta$ and $x_n>1$ imply that
$
-1-(\lambda x_n+a)<x_n-1.
$
In other words, $0<-1-x_{n+1}<x_n-1$ and, similarly,
$
0< x_{n+2}-1< -1-x_{n+1}.
$
Therefore, this trajectory would converge to the $2$-periodic orbit, which
does
not exist in this case. This contradiction implies that every trajectory
converges to an equilibrium point.

\subsubsection{$0\le \beta<1$}\label{Case:5}

In this case,
$a>0$ and so
the slope of the segment $EF$ is greater than 1 (see Fig.~\ref{fig:2j}). If a
trajectory starts in $\Pi$, then it converges to an
equilibrium point  $P^*\in EF$ along the line $p=const$.

A trajectory starting to the right of the parallelogram $\Pi$ moves along the
line $p=const$ down and left until it reaches the line $s=-1$. At this point,
or at the next iteration step, the trajectory reaches a point $(x_n,-1)$
that lies to the left of the equilibrium point $F$ because $\lambda<0$. If
$(x_n,-1)\in \Pi$, then the trajectory converges to an equilibrium as we
have seen above. If the point $(x_n,-1)$ lies to the left of the parallelogram $\Pi$,
then let us show that
the absolute value $|x_n-s_n|=-1-x_n$ of the $p$-coordinate of this point is less than the absolute value $|x_{n-1}-s_{n-1}|=x_{n-1}-s_{n-1}$ of the
$p$-coordinate of its preimage $(x_{n-1},s_{n-1})$.
Since $x_{n}=\lambda x_{n-1} +as_{n-1}$ and
$s_{n}=-1$, we want to show that
\begin{equation}\label{nbnbnb}
-(\lambda x_{n-1}+as_{n-1}+1)<(x_{n-1}-s_{n-1})q
\end{equation}
with some $q\in(-\lambda,1)$ independent of the point
$(x_{n-1},s_{n-1})$. Equivalently,
\[
x_{n-1}-s_{n-1}>\frac{-\beta s_{n-1}-1}{q+\lambda}.
\]

Indeed, since the point $(x_{n-1},s_{n-1})$ lies to the right of $\Pi$, we
have $x_{n-1}-s_{n-1}>\frac{-a}{1-\lambda}+1$ and it remains to show
that
\[
\frac{-a}{1-\lambda}+1>\frac{-\beta s_{n-1}-1}{q+\lambda}.
\]
This inequality is equivalent to
\begin{equation}\label{Eqn:1}
1+q>-\beta s_{n-1}(1-\lambda)+\beta (q+\lambda).
\end{equation}
If we set $q=-\lambda+\varepsilon$ with a sufficiently small
$\varepsilon>0$, then $q\in(0,1)$ and the inequalities $|s_{n-1}|\leq1$ and
$0\leq\beta<1$ imply that
\[
-\beta s_{n-1}(1-\lambda)< 1-\lambda,\qquad
\beta (q+\lambda)=\varepsilon( \lambda+a)<\varepsilon,
\]
hence the relation \eqref{Eqn:1} holds.

Since $q$ in \eqref{nbnbnb} does not depend on $(x_{n-1},s_{n-1})$ and
the segment connecting 
the points $P_1=(1,1)$ and $P_2=(-1,-1)$ belongs to the interior of $\Pi$, all
trajectories that start outside
the
parallelogram $\Pi$ will eventually enter $\Pi$ and converge to one of the
equilibrium points.


\subsection{Case \ref{item_e}}\label{Case:8}
 In this case, $a>0$ and the slope of the segment
 $EF$ is positive and less than $1$ (see Fig. \ref{fig:2k}).
\paragraph*{1}
Denote by $l_1$ and $l_2$ the open half-lines starting at the point $E$ on
the upper boundary of the strip $L$:

\[
l_1=\left\{(x,s): x> \frac{a}{1-\lambda}, \ s=1\right\},
\qquad
l_2=\left\{(x,s): x< \frac{a}{1-\lambda}, \ s=1\right\}
\]
and by $l_3$ and $l_4$ the half-lines starting from the point $F$ on the lower boundary of
the strip $L$:
\[
l_3=\left\{(x,s): x< -\frac{a}{1-\lambda}, \ s=-1\right\},
\qquad
l_4=\left\{(x,s): x> -\frac{a}{1-\lambda}, \ s=-1\right\}.
\]

From the condition $\lambda<0$ it follows that $f(l_2) \subseteq l_1$ and
$f(l_4)\subseteq l_3$.  Also from Lemma \ref{lm:mon_EF} it follows that for
any point $(x_n,s_n)$ such
that $x_n>\frac{as_n}{1-\lambda}$ one has $x_{n+1}<x_n$.
Thus, starting from $l_1$, any trajectory  arrives after  finitely many
iterations to the
closed half-line $\overline{l_3}$.
%
Hence, we can define the first-hitting map ${ \mathcal P}: l_1\to
\overline{l_3}$
as ${\mathcal P}(A)=f^k(A)$ where $f^{k}(A)\in \overline{l_3}$ and
$f^{i}(A)\not
\in \overline{l_3}$ for $i=0,\ldots,k-1$.
This map can be represented by the scalar function
 $T: (\frac {a}{1-\lambda},\infty)\to [\frac {a}{1-\lambda},\infty)$ defined by the
 formula
\begin{equation}
(-T(x),-1)={\mathcal P}(x,1),\qquad x\in\left(\frac
{a}{1-\lambda},\infty\right).
\end{equation}
It is convenient to set $T(\frac {a}{1-\lambda})=\frac{a}{1-\lambda}$ and consider $T$ as a map of the half-line
$[\frac{a}{1-\lambda},\infty)$ into itself.

\paragraph*{2} In this part, we describe the structure of the function
$T(x)$.
%
We begin with the following observation.

\begin{lemma}
	\label{lm:x_m}
	 Let  $(x_0,1)\in l_1$ be a point such that the first $n-1$ iterations of it
	 under the map $f$ belong to the line $p=const$. Then for any $m\leq
	 n$ we have
\begin{equation}\label{n1}
x_m=\beta^mx_0-a(x_0-1)\sum_{i=0}^{m-1}\beta^i.
\end{equation}
\end{lemma}
\begin{proof}
For $m=1$ equation \eqref{n1} is obvious. Suppose that \eqref{n1} holds
for $m<n$.
Then
\[
x_{m+1}=\lambda x_m+a(1+x_m-x_0)
=\beta^{m+1}x_0-a(x_0-1)\sum_{i=0}^{m}\beta^i.
\]\end{proof}

 Let us show that for any $k\in \mathbb{N}$ there exists a unique point
 $(r_k,1)\in l_1$ such that its first $k-1$ iterations under the map $f$
 belong to the line $p=const$ and its $k$-th iteration is $(r_k-2,-1)$.
 Setting $m=k$, $x_0=r_k$, and $x_{m}=r_k-2$ in \eqref{n1}, we obtain
\begin{equation}\label{n5}
r_k=\frac{2+a\frac{1-\beta^k}{1-\beta}}{1-\beta^k+a\frac{1-\beta^k}{1-\beta}}
\end{equation}
(it is easy to see that the denominator does not vanish as long as $f^{i}(r_k,1)$ belongs to the interior of $L$ for $i=1,...,k-1$).
\begin{figure}[h]
\begin{center}
\includegraphics[width=0.5\textwidth]{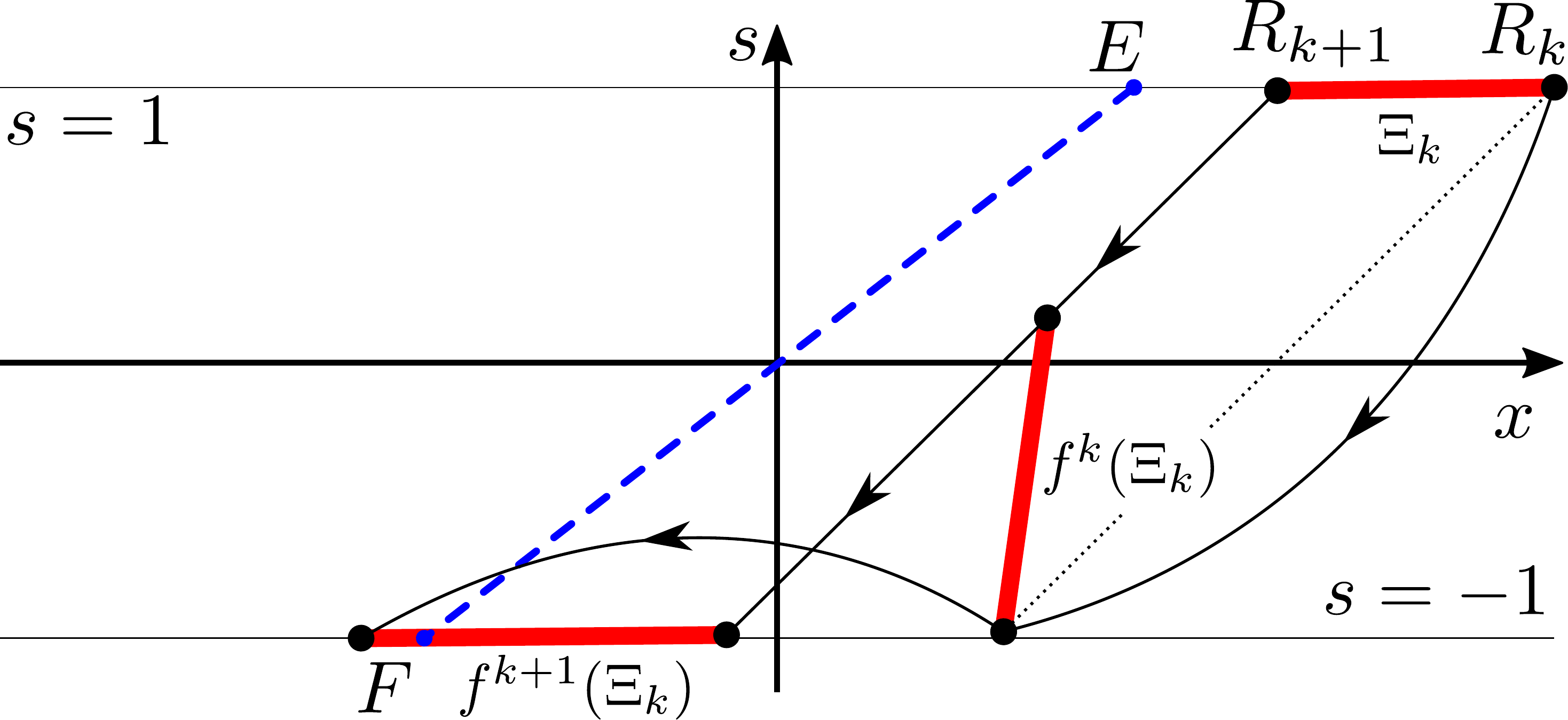}
\caption{ The segment $\Xi_k$ and its images $f^k(\Xi_k)$ and $f^{k+1}(\Xi_k)$.} \label{fig:2k}
\end{center}
\end{figure}

  Next we show that for any $k\in \mathbb{N}$ there exists a unique point
  $(q_k,1)\in l_1$ such that its first $k-1$ iterations under the map $f$
  belong to the line $p=const$ and its $k$-th iteration is $F$.

  Obviously, $q_{1}=\frac{1}{\lambda}(-a-\frac{a}{1-\lambda})$. Set
  $R_i=(r_i,1)$ and consider the $k$-th iteration of the segment
  $\Xi_k=R_{k+1}R_k$. The point $R_k$ is mapped  to the point
  $(r_k-2,-1)$,
  and the image of the point $R_{k+1}$ belongs to the interior of $L$.
  Hence, $f^k(\Xi_k)$ is a segment which lies entirely to the right of the
  segment $EF$ and all its points
  except $f^k(R_k)$ belong to the interior of $L$. Consider the $(k+1)$-st
  iteration of $\Xi_k$. The point $f^{k+1}(R_k)$ lies on the line $s=-1$ to
  the left of the point $F$, while $f^{k+1}(R_{k+1})=(r_{k+1}-2,-1)$. Hence,
  $f^{k+1}(\Xi_k)$ is a segment on the line $s=-1$ and $F\in
  f^{k+1}(\Xi_k)$, see Fig.~\ref{fig:2k}. Hence, there exists a point
  $q_{k+1}\in(r_{k+1},r_k)$ for each $k\in \mathbb{N}$. Now $q_k$ can be
  found in a unique way  using Lemma \ref{lm:x_m} by setting
$m=k$, $x_0=q_k$, and $x_{m}=\frac{-a}{1-\lambda}$:
%
%
\begin{equation}\label{n3}
q_k=\frac{-\frac{a}{1-\lambda}-a\frac{1-\beta^k}{1-\beta}}{\beta^k-a\frac{1-\beta^k}{1-\beta}}.
\end{equation}
Since $a>0$ and $\beta>1$, the denominator does not vanish.
%
%
%
%
Note that
\begin{equation}\label{points}
q_1>r_1>q_2>r_2>\cdots>\frac{a}{1-\lambda},\qquad q_k,r_k\to \frac{a}{1-\lambda} \ \ {\rm as} \ \ k\to\infty.
\end{equation}
It follows from the relation $(r_k-2,-1)=f^{k}(R_k)$ that $f^{k+1}(R_k)=(\lambda(r_k-2)-a,-1)$, i.e.,
\begin{equation}\label{eqTpk}
T(r_k)=a-\lambda(r_k-2).
\end{equation}
Combining~\eqref{points} and~\eqref{eqTpk}, we see that
\begin{equation}\label{eqTpkMonotone}
  T(r_1)>T(r_2)>\dots,\qquad T(r_k)-\dfrac{a}{1-\lambda}\to T_*>0,
\end{equation}
where
\begin{equation}\label{eqT*}
  T_*=\dfrac{2\lambda(1-a-\lambda)}{1-\lambda}.
\end{equation}
Furthermore, the above argument shows that the function $T(x)$ is continuous and piecewise linear for $x
\in(a/(1-\lambda),\infty)$. Using~\eqref{n5}, \eqref{n3}, and~\eqref{eqTpk}, we see that
it increases on the intervals $(q_1,
 \infty)$ and $(q_k,r_{k-1})$, $k=2,3,\dots$ with
\begin{equation}\label{eqT'qkqk-1}
T'(x)=-\lambda,\ x\in (q_1,\infty),\quad T'(x)=-\left(\beta^k-a\frac{1-\beta^k}{1-\beta}\right),\  x\in (q_k,r_{k-1}),
\end{equation}
 respectively, and
  decreases on the intervals $(r_k,q_k)$, $k\in\mathbb N$, with
\begin{equation}\label{eqT'pkqk2}
T'(x)=-\lambda\left(\beta^k-a\frac{1-\beta^k}{1-\beta}\right),\quad x\in (r_k,q_k),
\end{equation}
see Fig.~\ref{fig:2l}. Every point $x\in[q_1,\infty)$ possesses the property
that  $f(x,1)\in l_3$. Every point $x\in[q_{k+1},q_k)$, $k\in\mathbb N$,
possesses the property that the $(k+1)$th iteration of the point $(x,1)\in
l_1$ under the map $f$ reaches the half-line $l_3$ for the first time.



\begin{figure}[h]
\begin{center}
\includegraphics[width=0.65\textwidth]{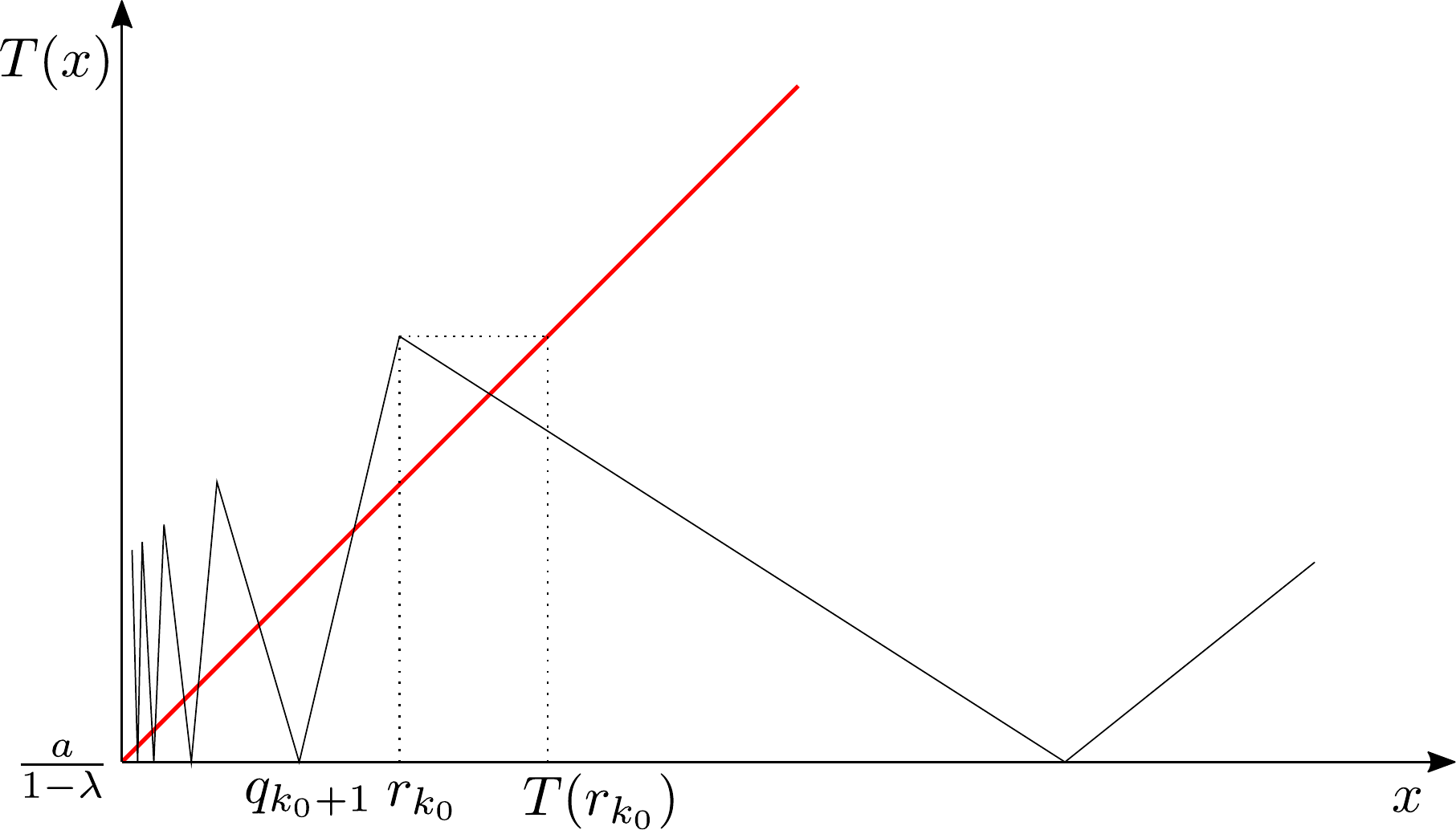}
\caption{Graph of the map $T(x)$ for
$x\in{\left(\frac{a}{1-\lambda},\infty\right)}$.} \label{fig:2l}
\end{center}
\end{figure}

\paragraph*{3} Now we show that system \eqref{eq2} has periodic
orbits
of all sufficiently large periods. Fix $k_1\in \mathbb{N}$ such that
\begin{equation}\label{eqqT*}
  q_{k_1}-\frac{a}{1-\lambda}<T_*,
\end{equation}
where $T_*>0$ is given by~\eqref{eqT*}.
 Note that $T([r_{k_1},q_{k_1}])=[\frac{a}{1-\lambda},T(r_{k_1})]$. Fix $k_2\in\mathbb{N}$ such that $$q_{k_2}<T(r_{k_1}).$$
For any $m\geq k_2$, denote by $\Theta_m$ the subsegment of
$[r_{k_1},q_{k_1}]$ such that $T(\Theta_m)=[r_m,q_m]$. It follows
from~\eqref{eqTpkMonotone} and~\eqref{eqqT*} that
$$
T^{2}(\Theta_m)\supset\left[\dfrac{a}{1-\lambda},\dfrac{a}{1-\lambda}+T_*\right]\supset
 \Theta_m.
$$
 Hence, the map $T^2$ has a fixed point in $\Theta_m$. Due to the
 argument in part 2 of this section, the corresponding periodic solution of
 the system~\eqref{eq2} will be of period $k_1+m+2$. Hence, for any
 $k\geq k_1+k_2+2$, system \eqref{eq2} has $k$-periodic orbit.

\paragraph*{4} To complete the proof of statement~\ref{item_e}, it remains to show that system  \eqref{eq2} has
no more than one stable periodic orbit. 
In this part, we find a necessary and sufficient condition for the map $T(x)$
to have fixed points in the interval $(q_{k+1},q_k)$ (obviously, $T(x)$ has
no fixed points for $x\ge q_1$ because $T'(x)=-\lambda\in(0,1)$ for all
$x>q_1$). Then we show that at most one fixed point of $T(x)$ can be
stable. Finally, in part $5$, we prove that all the
periodic orbits of $T(x)$ with minimal period greater than $1$ are unstable.

\begin{lemma}
	\label{lm:stab}
	The map $T(x)$ has a fixed point in the interval $(q_{k+1},q_k)$ if
	and only if
\begin{equation}\label{n8}
1+\lambda\beta^{k}\leq 0.
\end{equation}
The period of the corresponding orbit of
system~\eqref{eq2} equals $(2k+2)$.
\end{lemma}
\begin{proof}
The interval $(q_{k+1},q_k)$ contains a fixed point if and only if
\begin{equation}\label{n6}
T(r_k)\geq r_k,
\end{equation}
see Fig.~\ref{fig:2l}.
Using~formulas~\eqref{n5} and~\eqref{eqTpk}, we see that \eqref{n6} is equivalent to
\[
\frac{a+2\lambda}{1+\lambda}\geq\frac{2+a\frac{1-\beta^k}{1-\beta}}{1-\beta^k+a\frac{1-\beta^k}{1-\beta}},
\]
which can be rewritten as~\eqref{n8}.
\end{proof}

Note that, given $a$ and $\lambda$, inequality~\eqref{n8} holds for all
sufficiently large $k$. We denote by $k_0=k_0(\lambda,a)$ the smallest
$k$ with this property.
%
%

Now we fix $a$ and $\lambda$ and an arbitrary $k\ge k_0(\lambda,a)$ and
study  the stability of the fixed points of $T(x)$ in the interval
$(q_{k+1},q_k)$. First, note that if~\eqref{n8} holds as an equality, then the
interval $(q_{k+1},q_k)$ contains a unique fixed point $r_k$. It is unstable
because the slope of the graph of $T(x)$, $x\in(q_{k+1},r_k)$ is positive
and greater than one. Assume that~\eqref{n8} holds as a strict inequality,
i.e,
\begin{equation}\label{n8Strict}
  1+\lambda\beta^{k}<0.
\end{equation}
 Then there are two fixed points on the interval $(q_{k+1},q_k)$. The left one belongs to the interval $x\in(q_{k+1},r_k)$ and is unstable (as in the previous case). The right one belongs to the interval $(r_k,q_k)$. It is stable if and only if
\begin{equation}\label{eqT'pkqk1}
T'(x)\in [-1,0),\quad x\in (r_k,q_k).
\end{equation}
Combining~\eqref{eqT'pkqk1} and~\eqref{eqT'pkqk2}, we conclude that the fixed point from the interval $(r_k,q_k)$ is stable if and only if
\begin{equation}\label{n11}
  \lambda\beta^k+a-1\ge 0.
\end{equation}
\begin{lemma}
	\label{lm:omega_k}
	 Inequalities \eqref{n8Strict} and \eqref{n11} are either
	 incompatible for all $k\ge k_0$, or they hold for $k=k_0$ only.
\end{lemma}
\begin{proof}
In the
parameter plane $(\lambda,\beta)$, we introduce the regions $\Omega_k$
(see \eqref{eqn:C66'}).
 Assume
that for given $(
 \lambda,\beta)$, inequalities \eqref{n8Strict} and \eqref{n11} hold for
 $k_1$ and $k_2$. Then,  $(\lambda,\beta)\in
 \Omega_{k_1}\cap\Omega_{k_2}$. But $\Omega_{k_1}$ and
 $\Omega_{k_2}$ do not intersect for $k_1\ne k_2$. Hence, there exists
 no
 more than one $k
 \ge k_0$ for which inequalities \eqref{n8Strict} and \eqref{n11} hold simultaneously. Assume that both inequalities hold for some $k>k_0$. Then inequality \eqref{n8Strict} holds also for $k_0$ by definition of $k_0$, while inequality \eqref{n11} holds for $k_0$ due to the monotonicity of its left-hand side with respect to $k$. However, as we have just seen, both inequalities \eqref{n8Strict} and \eqref{n11} cannot hold for two different values $k$ and $k_0$ simultaneously.
\end{proof}

\paragraph*{5} We denote by $r_*$ the (unique) fixed point of the map
$T(x)$  in the interval $[r_{k_0},q_{k_0})$, where $k_0=k_0(\lambda,a)$ was
introduced in part 4. Recall that the map $T(x)$ has no fixed points for
$x>r_*$. In particular, this implies that
\begin{equation}\label{eqTxx}
  T(x)<x\quad \text{for all } x>r_*.
\end{equation}

It remains to show that all the fixed points of any iteration of $T(x)$, except for the  possibly stable fixed point $r_*$, are unstable.

First, we note that the fixed points of any iteration belong to the segment $[\frac{a}{1-\lambda},T(r_{k_0})]$.  Indeed, if $x>T(r_{k_0})$, then $x> r_*$, and~\eqref{eqTxx} implies that either  $T^j(x)\in [\frac{a}{1-\lambda},T(r_{k_0})]$ for some $j$, or $T^j(x)\to T(r_{k_0})$ as $j\to\infty$.

Second, let us show that the segment $[\frac{a}{1-\lambda},T(r_{k_0})]$ is mapped onto itself under $T(x)$. Indeed, $[q_{k_0+1},r_{k_0}]\subset [\frac{a}{1-\lambda},T(r_{k_0})]$ and $T([q_{k_0+1},r_{k_0}])=[\frac{a}{1-\lambda},T(r_{k_0})]$. Hence, $[\frac{a}{1-\lambda},T(r_{k_0})]\subset T([\frac{a}{1-\lambda},T(r_{k_0})])$. On the other hand, $T(x)\le T(r_{k_0})$ for $x\le r_*$ due to the monotonicity of $T(r_k)$ (see~\eqref{eqTpkMonotone}) and $T(x)<x\le T(r_{k_0})$ for $x\in[r_*, T(r_{k_0})]$ due to~\eqref{eqTxx}. Thus,
$$
T\left(\left[\frac{a}{1-\lambda},T(r_{k_0})\right]\right)=\left[\frac{a}{1-\lambda},T(r_{k_0})\right],
$$
and it suffices to study the iterations of $T(x)$ on the segment
$[\frac{a}{1-\lambda},T(r_{k_0})]$. We consider two cases: fixed point
$r_*$ being stable or unstable.

\subparagraph*{5.1} Assume that the fixed point $r_*$ is stable. Let us
show
that
\begin{equation}\label{eqTpk0pk0qk0}
  T(r_{k_0})\in [r_*,q_{k_0}).
\end{equation}
Obviously $T(r_{k_0})\ge r_*$. On the other hand, since
\begin{equation}\label{eqT'1pk0qk0}
|T'(x)|\le 1\quad \text{for all } x\in (r_{k_0},q_{k_0}),
\end{equation}
 it follows that
$$
q_{k_0}-r_{k_0}\ge T(r_{k_0})-\dfrac{a}{1-\lambda}> T(r_{k_0})- r_{k_0},
$$
i.e., $T(r_{k_0})\le q_{k_0}$.

Next, we show that the segment $[r_{k_0},T(r_{k_0})]$ is invariant under $T$, i.e.,
\begin{equation}\label{eqpk0Tpk0Invariant}
  T([r_{k_0},T(r_{k_0})])\subset [r_{k_0},T(r_{k_0})].
\end{equation}
 Since $T(x)$ is linear on this segment, we need to check the images $T(r_{k_0})$ and $T^2(r_{k_0})$ of the end points only. Obviously, $T(r_{k_0})$ belongs to this segment. Moreover, relations~\eqref{eqTpk0pk0qk0} and~\eqref{eqT'1pk0qk0} show that all the iterations of $r_{k_0}$ under the map $T$ belong to the segment $[r_{k_0},T(r_{k_0})]$ (and converge to $r_*$). In particular, $T^2(r_{k_0})$ belongs to this interval.

Now we are ready to prove that the fixed points of any iteration of $T(x)$, except for $r_*$, are unstable. Assume, to the contrary, that $x_*\in (\frac{a}{1-\lambda},T(r_{k_0})]$ is a stable fixed point of $T^{j_*}(x)$ for some $j_*\ge 2$ and $x_*\ne r_*$. We have seen in part 4 of this section that $|T'(x)|>1$ for all $x\in(r_{k+1},q_{k+1})\cup(q_{k+1},r_k)$, $k\ge k_0$. Therefore, the only possibility for $x_*$ to be stable is that $T^j(x_*)\in [r_{k_0},T(r_{k_0})]$ for some $j\in\mathbb N$. However, all the trajectories entering this segment converge to  $r_*$ due to~\eqref{eqT'1pk0qk0} and~\eqref{eqpk0Tpk0Invariant}.

\subparagraph*{5.2} Assume that the fixed point $r_*$ is unstable. Then
$k_0>1$ (otherwise, $|T'(r_*)|=\lambda^2<1$). It follows from the
monotonicity of $T(r_k)$ (see~\eqref{eqTpkMonotone}) and~\eqref{eqTxx}
that  $T(r_{k_0})< T(r_{k_0-1})<r_{k_0-1}$, i.e.,
\begin{equation}\label{eqTpk0pk0-1}
  \left[\dfrac{a}{1-\lambda},T(r_{k_0})\right]\subset \left[\dfrac{a}{1-\lambda},r_{k_0-1}\right].
\end{equation}
But $|T'(x)|>1$ for all $x\in (q_{k+1},r_k)\cup (r_k,q_k)$, $k\ge k_0$, due
to part 4 of this section, and $|T'(x)|=|\lambda^{-1} T'(r_*)|>1$ for  all
$x\in (q_{k_0},r_{k_0-1})$ due to~\eqref{eqT'qkqk-1}
and~\eqref{eqT'pkqk2}. This and~\eqref{eqTpk0pk0-1} imply that the
absolute values of all the slopes of the graph of $T(x)$ on the interval
$[\frac{a}{1-\lambda},T(r_{k_0})]$ are greater than 1. Hence, the same is
true for any iteration of $T(x)$ on this interval. Therefore, all the fixed
points of any iteration of $T$ are unstable. This completes the proof of
statement~\ref{item_e} of
Theorem \ref{Theorem:1}.

\subsubsection*{Proof of Theorem \ref{Theorem:2}}

If
$(\lambda,\beta)\in\Omega_k$
for some $k\in\mathbb N$, then, according to Lemmas \ref{lm:stab} and
\ref{lm:omega_k} , the
map $T$ has a stable fixed point in the interval $[r_k,q_k)$ and
$k=k_0(\lambda,a)$. It corresponds to the stable $(2k+2)$-periodic orbit
of system~\eqref{eq2}. According to part 5.1 of this section, all the other
periodic orbits are unstable.

If $(\lambda,a)\notin\bigcup_{k\in\mathbb N} \Omega_k$, then, according
to Lemma \ref{lm:omega_k}, the map $T$ has no stable fixed points.
Therefore,
due to part 5.2 of this section, all the periodic orbits are unstable.
\endproof

\subsection{Case \ref{item_f}}\label{Case:9}

In this case, the parallelogram $\Pi$ degenerates into the segment of the
equilibrium points $EF$ with the slope $1$ (see Fig.~\ref{fig:2m}).
\begin{figure}[h!]
	\begin{center}
		\includegraphics[width=0.5\textwidth]{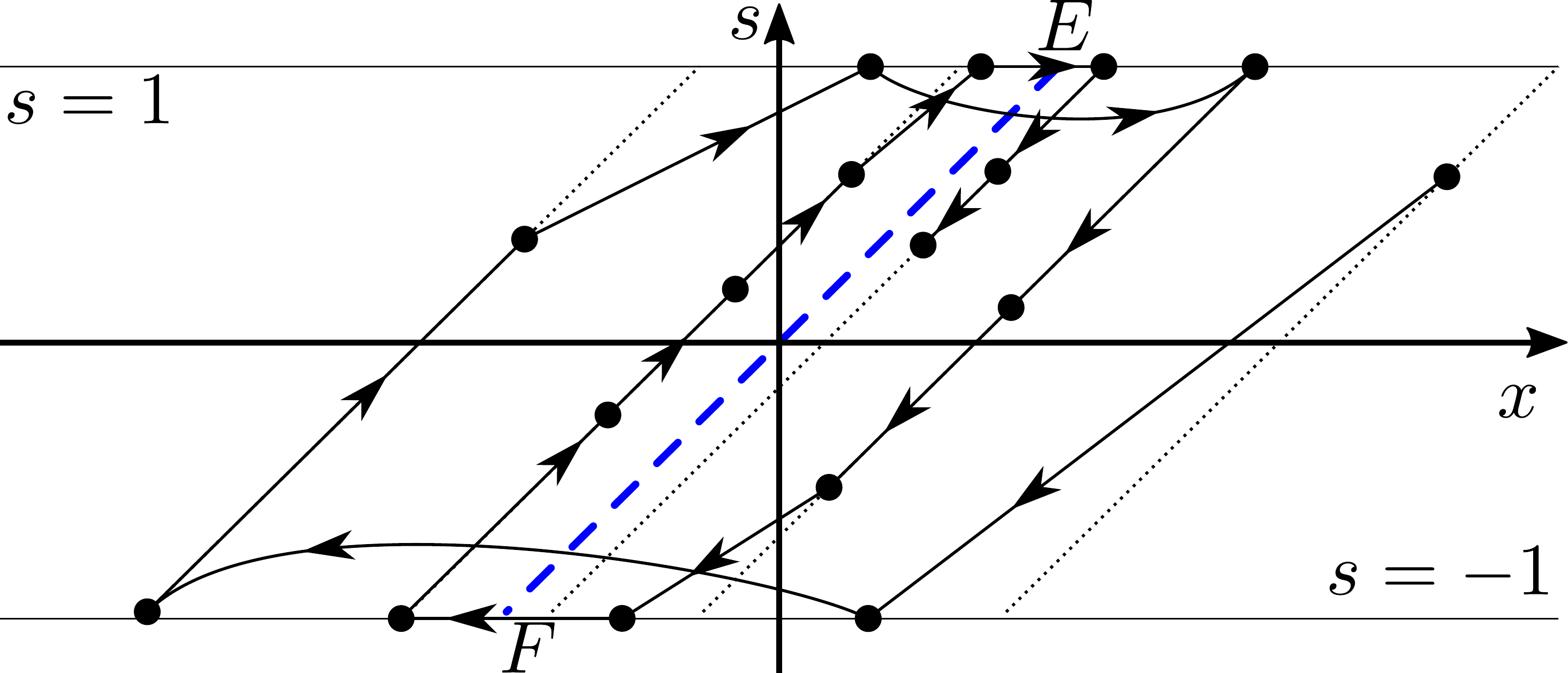}
		\caption{Case $\beta=1$, $\lambda<0$. The parallelogram $\Pi$
		degenerates to the segment $EF$ with slope
		$1$.} \label{fig:2m}
	\end{center}
\end{figure}

Let us consider a point $(x_n,s_n)\not\in EF$. To be definite, assume that
$p_n=x_n-s_n>0$. Denote by $(x_{n+k},s_{n+k})$ the first iteration that reaches the line $s=-1$ after the moment $n$,
i.e., $s_{n+i}>-1$ for $i=0,\ldots,k-1$ and $s_{n+k}=-1$ (if $s_{n}=-1$ we agree that $k=0$).
If $p_{n+k}=0$, then the trajectory ends at the point $F$. If $p_{n+k}>0$, then $0<p_{n+k}\le p_{n+k-1}=\cdots =p_n$ and
\begin{equation}\label{fix}
p_{n+k+1}=\lambda x_{n+k} +a s_{n+k} -s_{n+k+1}=\lambda x_{n+k} -a+1=\lambda x_{n+k}+\lambda=\lambda p_{n+k},
\end{equation}
where we use $a=1-\lambda$. If $p_{n+k}<0$, then $p_{n+k}<0< p_{n+k-1}=\cdots =p_n$ and
$$
p_{n+k}=\lambda x_{n+k-1} +a s_{n+k-1} -s_{n+k}=\lambda p_{n+k-1} +s_{n+k-1} -s_{n+k}\ge \lambda p_{n+k-1},
$$
hence
\begin{equation}\label{fixx}
|p_{n+k}|\le |\lambda| |p_{n+k-1}|.
\end{equation}
Inequalities \eqref{fix} and \eqref{fixx} and similar inequalities that hold for
ascending parts of trajectories, due to \eqref{eq:f_odd}, show that the
trajectory either ends up at $E$ or $F$, or converges to the segment $EF$.

\subsection{Case \ref{item_g}}\label{Case:10}
For $\beta=-1$ it is straightforward to see that the parallelogram
$\Sigma$, which is contained in $\Pi$, consists of $2$-periodic orbits and
the
segment $EF$ of equilibrium points.

If $\lambda\ge 0$ (see Fig.~\ref{fig:2n1}), then from Lemma \ref{lm:Pi_par}
it follows  that the parallelogram $\Pi$ is invariant under the map $f$
and if $(x_n,s_n)\in \Pi\setminus \Sigma$, then either $(x_{n+1},s_{n+1})\in\Sigma$
or $x_{n+1}<-1$, $s_{n+1}=-1$ or $x_{n+1}>1$, $s_{n+1}=1$. But, since $\lambda\ge0$,
relations $x_{n+1}<-1$, $s_{n+1}=-1$ and $f(-1,-1)=(1,1)$ imply $x_{n+2}<1$, $s_{n+2}=1$
and, similarly, relations $x_{n+1}>1$, $s_{n+1}=1$ and $f(1,1)=(-1,-1)$ imply $x_{n+2}>-1$, $s_{n+2}=-1$.
In both cases, $(x_{n+2},s_{n+2})\in \Sigma$. Thus, $f^{2}$ maps $\Pi$ into $\Sigma$.
On the other hand, the argument presented in Section \ref{Case:4} shows that a trajectory starting outside $\Pi$
either converges to the point $F$ along the line $s=-1$ from the right
or to the point $E$ along the line $s=1$ from the left or meets the boundary of the strip $L$ inside $\Pi$.\\

\begin{figure}[h]
\centering
\begin{subfigure}{0.48\textwidth}
  \centering
  \includegraphics[width=\textwidth]{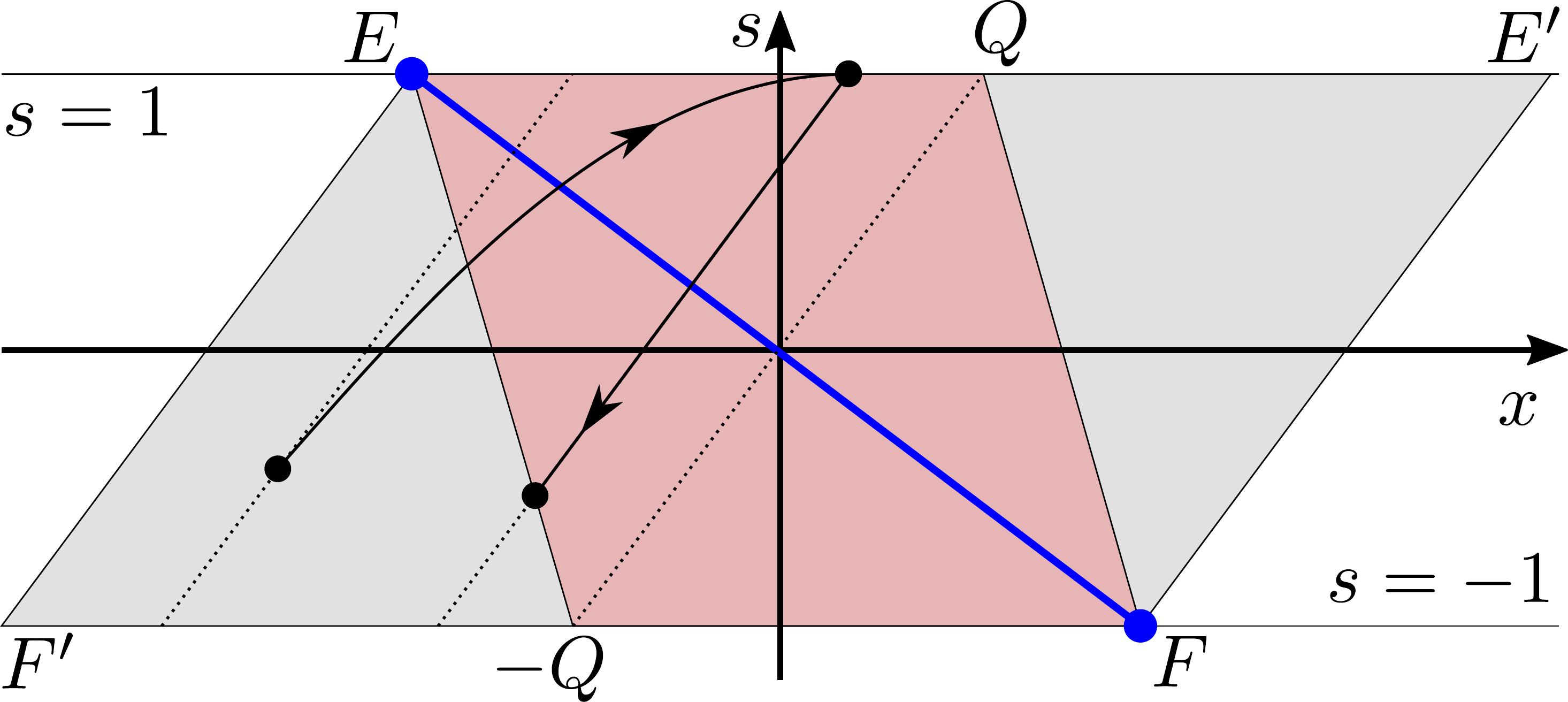}
  \caption{$\lambda\geq0$.}
  \label{fig:2n1}
\end{subfigure}\,
\begin{subfigure}{0.48\textwidth}
  \centering
  \includegraphics[width=\textwidth]{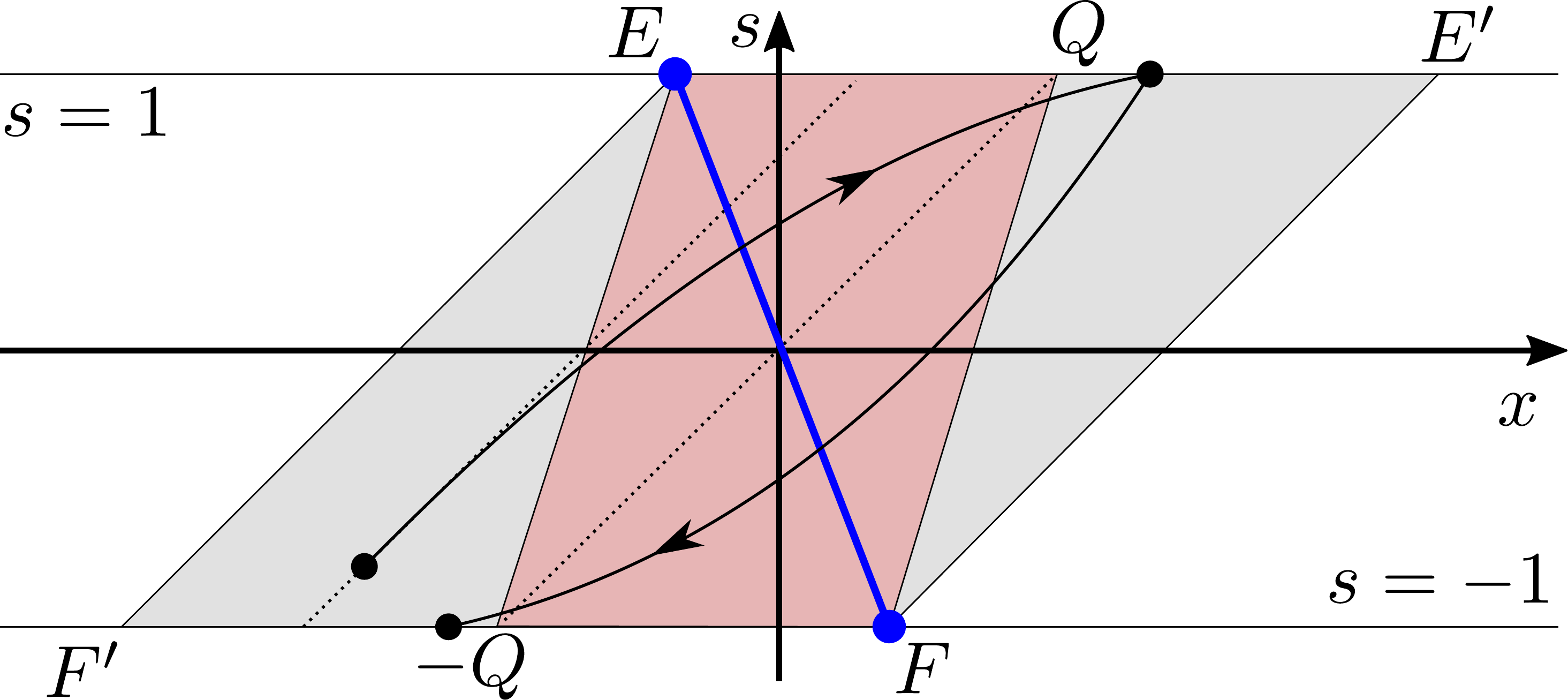}
  \caption{$\lambda<0$}
  \label{fig:2n2}
\end{subfigure}
\caption{Parallelograms $\Sigma$ and $\Pi$ for $\beta=-1$.}
\label{fig:2n}
\end{figure}

If $\lambda<0$ (see Fig.~\ref{fig:2n2}), then the argument used in Section
\ref{Case:7} shows that the set $M=\{(x,s):\, x\le -1,\, s=-1\}\cup\{(x,s):\,
x\ge 1,\, s=1\}$  is invariant under the map
$f$ and all the trajectories starting in $M$ converge to the $2$-periodic
orbit $(-1,-1)$, $(1,1)$.
If $(x_n,s_n)\in \Pi\setminus \Sigma$, then $(x_{n+1},s_{n+1})\in M\cup
\Sigma$ because $\beta=-1$.
Finally, if $(x_n,s_n)\not \in\Pi$, then $(x_{n+k},s_{n+k})\in M\cup
\Sigma$ for some $k\in \mathbb{N}$. This completes the proof of Theorem
\ref{Theorem:1}. \endproof


\subsection*{Acknowledgements}
NB acknowledges Saint-Petersburg State University
(research grant $6.38.223.2014$), the Russian Foundation for Basic Research
(project
No $16-01-00452$) and DFG project SFB $910$.
The work of PG was supported by the DFG Heisenberg
Programme and DFG project SFB $910$.
DR was supported by NSF grant DMS-$1413223$.


\bibliographystyle{plain}

\bibliography{stopbib}
\end{document}